\documentclass[12pt]{article}

\usepackage{amssymb,amsmath,amsthm}    
\usepackage[abbrev]{amsrefs}    
\usepackage{mathrsfs}
\usepackage[bbgreekl]{mathbbol}
\usepackage{comment}
\usepackage{color}
\usepackage{tikz}
\usepackage{mathtools}

\usepackage[a4paper, left=18mm, right=18mm, top=25mm, bottom=25mm]{geometry}

\newtheorem{theo}{Theorem}[section]
\newtheorem{lemm}[theo]{Lemma}
\newtheorem{corr}[theo]{Corollary}
\newtheorem{prop}[theo]{Proposition}
\newtheorem{conj}[theo]{Conjecture}

\numberwithin{equation}{section}

\theoremstyle{definition}
\newtheorem{defi}[theo]{Definition}
\newtheorem{exam}[theo]{Example}
\newtheorem{rema}[theo]{Remark}
\newtheorem{assu}[theo]{Assumption}
\newtheorem{claim}[theo]{Claim}

\allowdisplaybreaks

\newcommand{\can}{{\rm{can}}}
\newcommand{\rk}{{\rm{rk}}}
\newcommand{\Supp}{{\rm{Supp}}}
\newcommand{\Amp}{{\rm{Amp}}}
\newcommand{\sing}{{\rm{sing}}}
\newcommand{\Sing}{{\rm{Sing}}}
\newcommand{\Exc}{{\rm{Exc}}}
\newcommand{\reg}{{\rm{reg}}}
\newcommand{\Null}{{\rm{Null}}}
\newcommand{\codim}{{\rm{codim}}}
\newcommand{\Id}{{\rm{Id}}}
\newcommand{\Ker}{{\rm{Ker}}}
\newcommand{\Mov}{{\rm{Mov}}}

\usepackage[colorlinks=true, linkcolor=blue, citecolor=blue, urlcolor=cyan]{hyperref}

\begin{document}
\title{Slope Stable Sheaves and Hermitian-Einstein Metrics on Normal Varieties with Big Cohomology Classes}
\author{Satoshi Jinnouchi}
\date{\empty}
\maketitle
\begin{abstract}
\small{In this paper, we introduce the notions of slope stability and the Hermitian-Einstein metric for big cohomology classes. Our main result is the Kobayashi-Hitchin correspondence on compact normal spaces with big classes admitting the birational Zariski decomposition with semiample positive part. We also prove the Bogomolov-Gieseker inequality for slope stable sheaves with respect to nef and big classes. Through this paper, the ``bimeromorphic invariance'' of slope stability and the existence of Hermitian-Einstein metrics play an essential role.}
\end{abstract}

\section{Introduction, Main Result}
\subsection{Introduction}
This paper focuses on extending the Kobayashi-Hitchin (hereinafter abbreviated as KH) correspondence by generalizing the slope stability and the notion of Hermitian-Einstein (denoted as HE simply) metrics from compact Kähler manifolds to more general settings involving big cohomology classes on compact normal complex varieties.

The results build on fundamental works in complex and algebraic geometry, employing tools such as non-pluripolar products for closed positive (1,1)-currents.
The KH correspondence, originally established by Donaldson \cite{Don3} and Uhlenbeck-Yau \cite{UY} for compact Kähler manifolds, asserts that a holomorphic vector bundle over a compact Kähler manifold is slope polystable if and only if it admits a Hermitian-Einstein metric.
Bando and Siu extended the KH correspondence to reflexive sheaves by defining the Hermitian-Einstein metric on a Zariski open set that satisfies the admissible condition \cite{BS}.
More recently, Xuemiao Chen showed that the KH correspondence holds for compact K\"{a}hler normal varieties \cite{Chen}.

Building on these developments, this paper extends the notions of slope stability and Hermitian-Einstein metrics from Kähler classes to big classes, establishing their bimeromorphic invariance through the use of non-pluripolar products.
Furthermore, we prove the KH correspondence on both normal projective varieties of general type with log terminal singularities.
This approach leverages recent advancements by Boucksom, Eyssidieux, Guedj, and Zeriahi, who showed that the Monge-Ampère equation can be solved on compact complex manifolds with a big cohomology class \cite{BEGZ}.

These extensions not only broaden the applicability of the KH correspondence but also connect it with recent developments in pluripotential theory and birational geometry, potentially enriching our understanding of vector bundles and stability conditions on normal varieties.

\subsection{Main Result}
The Kobayashi-Hitchin correspondence is the equivalence between the slope stability and the existence of Hermitian-Einstein metrics on holomorphic vector bundles (more generally reflexive sheaves). On compact K\"{a}her manifolds, this correspondence is well-studied and proved by  \cite{Don3}, \cite{UY}, \cite{BS}, \cite{Chen}. 

In this paper, we extend the notions of slope stability, Hermitian-Einstein metrics and the Kobayashi-Hichin correspondence from K\"{a}hler classes to big classes.
Let $X$ be a compact complex manifold of dimension $n$, $\alpha$ be a big class on $X$ 
and $\mathcal{E}$ be a reflexive sheaf on $X$. 

\begin{defi}[= Definition \ref{stability defi}, Definition \ref{stability defi2}]
A reflexive sheaf $\mathcal{E}$ is $\langle\alpha^{n-1}\rangle$-slope stable if for any reflexive subsheaf $0\neq \mathcal{F}\subsetneq \mathcal{E}$, the inequality $\mu_{\alpha}(\mathcal{F})<\mu_{\alpha}(\mathcal{E})$ holds, where
$$
\mu_{\alpha}(\mathcal{E})=\frac{1}{\rk(\mathcal{E})}\int_Xc_1(\det\mathcal{E})\wedge\frac{\langle\alpha^{n-1}\rangle}{(n-1)!}.
$$
\end{defi}
Here $\langle\alpha^{n-1}\rangle$ is the positive product defined by Boucksom-Eyssidieux-Guedji-Zeriahi \cite{BEGZ} and Boucksom-Demailly-P\u{a}un-Peternell \cite{BDPP}. 
Since $\alpha$ is big, there is a closed positive $(1,1)$-current $T$ in $\alpha$ which is smooth K\"{a}hler on $\Amp(\alpha)$ the ample locus of $\alpha$. Set $\Omega:=(X\setminus\Sing(\mathcal{E}))\cap\Amp(\alpha)$.
\begin{defi}[= Definition \ref{THE defi}]
A $T$-Hermitian-Einstein metric is a smooth hermitian metric $h$ on $\mathcal{E}|_{\Omega}$ which satisfies
\begin{itemize}
\item $\sqrt{-1}\Lambda_TF_h=\lambda\cdot \Id$ on $\Omega$,
\item $\int_{\Omega}|F_h|^2_TT^n<\infty,$
\item $\lambda=\frac{n}{\int_X\langle\alpha^n\rangle}\int_Xc_1(\det\mathcal{E})\wedge\langle\alpha^{n-1}\rangle.$
\end{itemize}
\end{defi}
\noindent We will generalize these notions to compact normal spaces (see Definition \ref{stability normal defi}, Definition \ref{THE normal}).

Since the bigness of a cohomology class is bimeromorphic invariant, it is natural to expect that both $\langle\alpha^{n-1}\rangle$-slope stability and the existence of $T$-Hermitian-Einstein metrics are also bimeromorphic invariant. In order to establish the invariance, we impose the following assumption in this paper:
\begin{assu}[= Assumption \ref{assumption}, see also section \ref{Assumption}]\label{assumption intro}
Let $\pi:Y\to X$ be a bimeromorphic morphism between compact K\"{a}hler manifolds and $\alpha$ be a big class on $X$. Then we assume 
$$
\langle(\pi^*\alpha)^{n-1}\rangle\cdot[D]=0
$$
holds for any $\pi$-exceptional divisor $D$.
\end{assu}
\noindent If $Y$ is projective, Assumption \ref{assumption intro} was proven in \cite{Nystr19}. 
Under the Assumption \ref{assumption intro}, we obtain the bimeromorphic invariance of $\langle\alpha^{n-1}\rangle$-slope stability and the existence of $T$-Hermitian-Einstein metrics:
\begin{theo}[= Theorem \ref{stability bimero}, Theorem \ref{birational HE}]\label{invariance intro}
Let $X$ be a compact normal space, $\alpha$ be a big class on $X$ and $\mathcal{E}$ be a reflexive sheaf on $X$. Let $Y$ be a compact normal space, $\beta$ be a big class on $Y$ and $\mathcal{F}$ be a reflexive sheaf on $Y$. 
 Let $\pi:Y\dashrightarrow X$ be a bimeromorphic map satisfying
\begin{itemize}
\item $\pi_*\beta=\alpha$ and $\pi$ is $\beta$-negative contraction (e.g. a composition of flips and divisorial contractions) 
\item $\pi^{[*]}\mathcal{E}\simeq\mathcal{F}$ away from the $\pi$-exceptional locus.
\end{itemize}
Then we have the followings:
\begin{enumerate}
\item $\mathcal{F}$ is $\langle\beta^{n-1}\rangle$-slope stable if and only if the reflexive sheaf $\mathcal{E}$ is $\langle\alpha^{n-1}\rangle$-slope stable. 
\item The reflexive sheaf $\mathcal{E}$ admits a $T$-HE metric if and only if $\mathcal{F}$ admits a $T'$-HE metric. Here $T\in \alpha$ and $T'\in\beta$ are suitable closed positive $(1,1)$-currents.
\end{enumerate}
If $\alpha$ is nef and big or $X$ and $Y$ are projective, then we do not need Assumption \ref{assumption intro}.
\end{theo}

The main goal of this paper is to establish the Kobayashi-Hitchin correspondence on compact complex manifolds with big classes which admits a birational Zariski decomposition with semiample positive part (e.g. a semiample class).
The correspondence is an application of Theorem \ref{invariance intro}. To be more specific, we prove the following result. We do not need Assumption \ref{assumption intro}.
\begin{theo}[= Theorem \ref{KH corr}]\label{KH intro}
Let $X$ be a compact normal space with a big class $\alpha\in H^{1,1}_{BC}(X)$ and $\mathcal{E}$ be a reflexive sheaf on $X$. Suppose that $\alpha$ admits a birational Zariski decomposition with semiample positive part (e.g. $\alpha$ is semiample). Then $\mathcal{E}$ is $\langle\alpha^{n-1}\rangle$-slope stable if and only if $\mathcal{E}$ admits a $T$-Hermitian-Einstein metric with a suitable $T\in\alpha$.
\end{theo}
If $X$ is smooth and $\alpha$ is a K\"{a}hler class, the above theorem is the well-known Kobayashi-Hitchin correspondence proven in  \cite{Don3},\cite{UY} and \cite{BS}.  Xuemiao Chen proved in singular settings \cite{Chen}. As a direct consequence of Theorem \ref{invariance intro} and Theorem \ref{KH intro}, the Kobayashi-Hitchin correspondence of a projective variety of general type coincides with that of the canonical model (see Example \ref{cor-stability-tangent-big}, Example \ref{HE blowup}, Corollary \ref{KH 1'}).

\vskip\baselineskip
One of the important properties of slope stable sheaves is the Bogomolov-Gieseker inequality. 
If a reflexive sheaf is slope stable with respect to a K\"{a}hler class, the inequality has been extensively studied (c.f. \cite{UY},\cite{BS}, \cite{Chen}). The Bogomolov-Gieseker inequality is closely related to the Miyaoka-Yau inequality, whose equality case characterizes the uniformization of projective varieties \cite{GKPT}.

In this paper, we establish the Bogomolov-Gieseker inequality for reflexive sheaves that are slope stable with respect to a nef and big class on a compact normal space. We do not need Assumption \ref{assumption intro} for the following theorem.
\begin{theo}[= Proposition \ref{BG nef big}]\label{BG intro}
Let $X$ be a compact normal space with a nef and big class $\alpha \in H^{1,1}_{BC}(X, \mathbb{R})$. Let $\mathcal{E}$ be a reflexive sheaf on $X$ and $\pi : \widehat{X} \rightarrow X$ be a resolution of singularities. Let $\pi^{[*]}\mathcal{E}:= (\pi^{*}\mathcal{E})^{**}$ be the reflexive pullback of $\mathcal{E}$. Suppose $\mathcal{E}$ is $\alpha^{n-1}$-slope stable. Then, the following Bogomolov-Gieseker inequality holds:
$$
\left(2rc_{2}(\pi^{[*]}{\mathcal{E}}) - (r-1)c_{1}(\pi^{[*]}{\mathcal{E}})^{2}\right)\cdot(\pi^{*}\alpha)^{n-2}\geq 0.
$$
\end{theo}
As a corollary of Theorem \ref{BG intro}, we obtain the characterization of the equality case on minimal projective varieties of general type.
\begin{corr}[= Theorem \ref{BG equality}]\label{BG cor intro}
Let $X$ be a normal projective variety with log canonical singularities where $K_X$ is nef and big. Let $\mathcal{E}$ be a reflexive sheaf on $X$. Suppose $\mathcal{E}$ is $c_1(K_X)^{n-1}$-stable. If there exists a resolution $\pi:Y\to X$ such that $\pi^{[*]}\mathcal{E}$ satisfies the Bogomolov-Gieseker equality: $\Delta(\pi^{[*]}\mathcal{E})c_1(\pi^*K_X)^{n-2}=0$, then $\mathcal{E}$ is projectively flat on $\Amp(K_X)$. 
\end{corr}
\noindent As in Deinition \ref{nonKahler locus singular}, the nonK\"{a}hler locus $E_{nK}(\alpha)$ of a big class $\alpha$ on a compact normal space $X$ contains the singular locus $X_{\sing}$. In particular, if $\varphi : X \to X_{\can}$ be the birational morphism to the canonical model, then $\Amp(K_X)$ is mapped to $X_{\can, \reg}$.

\begin{rema}[Related work]
Let $\alpha$ be a big class on a compact K\"{a}hler manifold $X$. Then the positive product $\langle\alpha^{n-1}\rangle$ lies in the closed cone $\Mov_1(X)$ consisting of movable classes (c.f. \cite{BDPP}). There are studies of slope stability with respect to movable curves on projective manifolds.
Lehmann-Xiao \cite{LX15} studied the correspondence between movable classes and big divisors via the map defined by the positive product: $L\to \langle c_1(L)^{n-1}\rangle$.
Greb-Kebekus-Peternell \cite{GKP} studied slope stability with respect to movable class on projective manifolds, and proved the Bogomolov-Gieseker inequality of slope semistable sheaves with respect to movable classes on projective surface. There is a chamber decomposition of the ample cone of projective manifolds such that the moduli space of slope semistable sheaves does not change in each chamber. It is known that the pathological phenomena occur on walls of chambers. Greb-Toma \cite{GT17} solved the problem by considering the moduli space of slope semistable sheaves with respect to movable curves $(H_1\cdot\ldots\cdot H_{n-1})$ where each $H_i$ are ample. 
\end{rema}

The organization of this paper is as follows: 
Section \ref{preliminary} is devoted to review basic notions and preliminary results. We mainly deal with positive cohomology classes on smooth and normal spaces. 
Section \ref{Assumption} explains the Assumption \ref{assumption intro} we need in this paper. 
Section \ref{stability 0} introduces the notion of $\langle\alpha^{n-1}\rangle$-slope stability and the proof of Theorem \ref{invariance intro} (1). 
Section \ref{THE 0} defines the notion of $T$-Hermitian-Einstein metrics and proves Theorem \ref{invariance intro} (2). 
Section \ref{KH corre} deals with the Kobayashi-Hitchin correspondence, Theorem \ref{KH intro}.   
Section \ref{BG ineq} discusses the Bogomolov-Gieseker inequality and includes Theorem \ref{BG intro} and Corollary \ref{BG cor intro}.

\section*{Acknowledgment}
The author would like to greatly thanks to Prof. Ryushi Goto for many helpful comments and constructive discussions, and for pointing out many typos in the draft of this paper. The author would like to thanks to Masataka Iwai for discussion around the Bogomolov-Gieseker inequality.
The author would like to thanks to Prof. Brian Lehmann for pointing out several important, and interesting prior works. The author would like to thanks to Shiyu Zhang for pointing out my misunderstanding and fruitful discussion.
 
\section{Preliminary}\label{preliminary}
In this paper, we assume that any compact normal space is reduced, irreducible and second countable.
\subsection{positive cohomology class}\label{positive cohomology class}
There is a several notions of positivities of bidegree $(1,1)$ cohomology classes.
\begin{defi}[c.f. \cite{Dem2}, \cite{FT18}]\label{positive class}
Let $X$ be a compact K\"{a}hler manifold and $\alpha\in H^{1,1}(X,\mathbb{R})$.
\begin{enumerate}
\item We say $\alpha$ is pseudo-effective if $\alpha$ is represented by a closed positive $(1,1)$-current.
\item We say $\alpha$ is big if $\alpha$ is represented by a K\"{a}hler current. Here a K\"{a}hler current is a closed positive $(1,1)$-current $T$ on $X$ satisfying $T \geq \omega$ for some strictly positive $(1,1)$-form $\omega$.
\item We say $\alpha$ is  nef if, for any $\varepsilon > 0$, there is a smooth $(1,1)$-form $\alpha_{\varepsilon}$ in $\alpha$ which satisfies $\alpha_{\varepsilon} \geq -\varepsilon \omega$ where $\omega$ is a strictly positive $(1,1)$-form on $X$.
\item We say $\alpha$ is  semiample if there is a holomorphic surjection $\pi:X\to Y$ with connected fibres to a normal K\"{a}hler space $Y$ with a K\"{a}hler class $\omega\in H^{1,1}_{BC}(Y,\mathbb{R})$ such that $\alpha=\pi^*\omega$.
\end{enumerate}
\end{defi}
We will see that, if a semiample class $\alpha$  is also big, then the holomorphic surjection $\pi:X\to Y$ in the Definition \ref{positive class} is a bimeromorphic morphism (Proposition \ref{semiample-bimero}).

On a compact normal space, we use the Bott-Chern cohomology group to define positive cohomology classes. The readers can consult to \cite{HP16} for the definition of smooth differential forms and the Bott-Chern classes on singular spaces.
\begin{defi}[Definition 3.10 in \cite{HP16}]
Let $X$ be a compact normal space. Let $\alpha\in H^{1,1}_{BC}(X)$.
We say $\alpha$ is nef if, for any $\varepsilon>0$, there exists a smooth representative $\alpha_{\varepsilon}\in \alpha$ such that $\alpha_{\varepsilon}\ge -\varepsilon\omega$, where $\omega$ is a smooth strictly positive $(1,1)$-form on $X$.
\end{defi}
\begin{defi}\label{big normal}
Let $X$ be a compact normal space. A big class on $X$ is a cohomology class $\alpha \in H^{1,1}_{BC}(X, \mathbb{R})$ such that, for any resolution of singularities $f : \widehat{X} \rightarrow X$, the pull-back $f^{*}\alpha$ is a big class on $\widehat{X}$.
\end{defi}  
Das-Hacon-P\u{a}un \cite{DHP24} showed the following characterization of bigness and nefness via resolution of singularities:
\begin{lemm}[\cite{DHP24}, Corollary 2.32, Lemma 2.35]
Let $X$ be a compact normal space. Let $\pi:\widehat{X}\to X$ be a resolution of singularities of $X$. Then
\begin{enumerate}
\item $\alpha\in H^{1,1}_{BC}(X)$ is nef if $\pi^*\alpha\in H^{1,1}(\widehat{X},\mathbb{R})$ is nef,
\item $\alpha$ is big if there exists a K\"{a}hler current $T$ in $\alpha$.
\item A nef class $\alpha\in H^{1,1}_{BC}(X)$ is big in the sense of Definition \ref{big normal} if $(\pi^*\alpha)^n>0$.
\end{enumerate} 
\end{lemm}

\subsection{nonpluripolar product and positive product}\label{nonpluri-positive product}
Let $X$ be a compact K\"{a}hler manifold.
Let $T = \theta + dd^c\varphi$ be a closed positive $(1,1)$-current on $X$, where $\theta$ is a smooth closed $(1,1)$-form. In this paper, we only consider currents $T$ whose postential $\varphi$ has its unbounded locus contained in an analytic subvariety $V$ (i.e. $T$ has a small unbounded locus \cite{BEGZ}). In this case, the $nonpluripolar$ $product$ of $T$ is defined as $\langle T^p\rangle:=\mathbb{1}_{X\setminus V}T^p$ \cite{BEGZ}. Here $T^p$ on $X\setminus V$ denotes the product in the sense of Bedford-Taylor \cite{BT}. It is shown in \cite{BEGZ} that $\langle T^p\rangle$ is a closed positive $(p,p)$-current.
 
\hspace{-10.5pt}Let $\alpha$ be a pseudo-effective class on $X$ and $\theta \in \alpha$ be a smooth $(1,1)$ form.
Let $T = \theta + dd^{c}\varphi$ and $T' = \theta + dd^{c}\varphi'$ be closed positive $(1,1)$-currents in $\alpha$. We say that $T$ is less singular than $T'$ if $\varphi' \leq \varphi+f$ for some bounded function $f \in L^{\infty}(X)$. A closed positive $(1,1)$ current $T \in \alpha$ is said to have $minimal$ $singularities$, often denoted by $T_{\min}$, if $T_{\min}$ is less singular than any other closed positive $(1,1)$-current in $\alpha$ \cite{BEGZ}. Although a closed positive $(1,1)$-current with minimal singularities in $\alpha$ is not unique, the following proposition holds. 
\begin{prop}[\cite{BEGZ}]\label{positive product}
Let $\alpha$ be a big class and $T_{\min} \in \alpha$ be a closed positive $(1,1)$-current with minimal singularities. Then, for any closed positive $(1,1)$ current $T \in \alpha$, the inequality $\{\langle T^{p} \rangle\} \leq \{\langle T_{\min}^{p} \rangle\}$ holds for $p = 1, \ldots, n$.
In particular, the cohomology class $\langle \alpha^{p} \rangle := \{\langle T_{\min}^{p} \rangle\}$ for $p = 1, \ldots, n$ is uniquely determined by $\alpha$ and $p$.
\end{prop} 
The inequality of bidegree $(p,p)$ cohomology classes $\beta \geq \alpha$ means that the difference $\beta - \alpha$ is represented by a closed positive $(p,p)$-current. 
\begin{defi}[\cite{BEGZ}]\label{positive product defi}
Let $X$ be a compact K\"{a}hler manifold and $\alpha$ be a big class on $X$. The $positive$ $product$ of $\alpha$ is defined as $\langle\alpha^p\rangle:=\{\langle T_{\min}^p\rangle\}$, here $T_{\min}$ is a closed positive $(1,1)$-current with minimal singularities in $\alpha$.
\end{defi}
There is another algebraic notion of product, so called {\it{movable intersection product}} in \cite{BDPP}. M. Principato \cite{Mario} showed that movable intersection product coincides with positive product. For nef and big classes, the nonpluripolar product coincides with the usual intersection product:
\begin{prop}[see \cite{BDPP} Theorem 3.5 ({$\mathrm{ii}$})]\label{nef positive product}
Let $X$ be a compact K\"{a}hler manifold and $\alpha$ be a big and nef class on $X$. Then $\langle \alpha^{p} \rangle = \alpha^{p}$ holds for $p = 1, \cdots n$.
\end{prop}

A big class is characterized by the volume as follows.
\begin{prop}[\cite{BEGZ}]\label{big volume}
A pseudo-effective class $\alpha$ is big if $\langle \alpha^{n} \rangle \neq 0$.
\end{prop} 

\subsection{non-K\"{a}hler locus, non-nef locus and Null locus}\label{nonKahler-nef-null}
\begin{defi}[\cite{BEGZ}, see also \cite{Bou2}]\label{ample locus}
Let $X$ be a compact K\"{a}hler manifold and $\alpha$ be a big class on $X$. The $ample$ $locus$ of $\alpha$ is the subset of $X$ defined as follows,
\begin{equation}
\Amp(\alpha):= \{ x \in X \mid \text{There is a K\"{a}hler current $T \in \alpha$ smooth around $x$.\}}
\end{equation}
The complement $E_{nK}(\alpha) = X \setminus \Amp(\alpha)$ is the $non\text{-}K\ddot{a}hler$ $locus$ of $\alpha$.
\end{defi}
\hspace{-10.5pt}It is known that $\Amp(\alpha)$ is Zariski open and $E_{nK}(\alpha)$ is Zariski closed \cite{Bou2}. For any big line bundle $L$, we denote by $\Amp(L) := \Amp(c_{1}(L))$ and $E_{nK}(L) := E_{nK}(c_{1}(L))$. As illustrated in the following example, the subvariety $E_{nK}(\alpha)$ can be viewed as a generalization of the exceptional locus of a birational morphism.
\begin{exam}
Let $X$ be a projective manifold of general type, that is, the canonical divisor $K_{X}$ is big. Then, by the minimal model program (MMP) starting from $X$, which terminates in this case, there is a birational morphism $f : X \dashrightarrow X_{can}$ to the canonical model $X_{can}$ of $X$ \cite{BCHM}. In this setting, $E_{nK}(K_{X})$ is the exceptional set of $f$. In fact, $f$ is a composition of divisorial contractions and flips, a special type of a birational map isomorphic in codimension 1. Hence, $E_{nK}(K_{X})$ is the sum of the exceptional divisors of divisorial contractions and the exceptional sets of flips. 
\end{exam}
To the non-K\"{a}hler locus of big class on singular spaces, we need the following lemma. The readers can find a study of non-K\"{a}hler locus on singular spaces in \cite{HP24}.
\begin{lemm}\label{nonKahler locus singular lemma} 
Let $X$ be a compact normal space and $\alpha\in H^{1,1}_{BC}(X,\mathbb{R})$ be a big class on $X$. Let $\pi:Y\to X$ be a resolution of singularities of $X$. Then the set $\pi(E_{nK}(\pi^*\alpha))$ is independent of $\pi$.
\end{lemm}
\begin{proof}
Let $\pi_1:X_1\to X$ and $\pi_2:X_2\to X$ be two resolutions of $X$. We choose a common resolutions of $\pi_1$ and $\pi_2$ as follows:
\begin{center}
\begin{tikzpicture}[auto]
\node (x) at (1,0) {$X$}; \node (x1) at (0,1) {$X_1$}; \node (x2) at (2,1) {$X_2$}; \node (y) at (1,2) {$Y$};
\draw[->] (y) to node[swap] {$\mu_1$} (x1); \draw[->] (y) to node {$\mu_2$} (x2);
\draw[->] (x1) to node[swap] {$\pi_1$} (x); \draw[->] (x2) to node {$\pi_2$} (x);
\end{tikzpicture}
\end{center}
By \cite[Lemma 4.3]{CT22}, we know
$$
E_{nK}(\mu_i^*\pi_i^*\alpha)=\mu_i^{-1}(E_{nK}(\pi_i^*\alpha))\cup \Exc(\mu_i)
$$
for $i=1,2$. Since $\mu_1^*\pi_1^*\alpha=\mu_2^*\pi_2^*\alpha$, we also have $E_{nK}(\mu_1^*\pi_1^*\alpha)=E_{nK}(\mu_2^*\pi_2^*\alpha)$. Therefore we obtain
\begin{align*}
\pi_1\left(E_{nK}(\pi_1^*\alpha)\right)
&=(\pi_1\circ\mu_1)\left(E_{nK}(\mu_1^*\pi_1^*\alpha)\right)\\
&=(\pi_2\circ\mu_2)\left(E_{nK}(\mu_2^*\pi_2^*\alpha)\right)\\
&=\pi_2\left(E_{nK}(\pi_2^*\alpha)\right).
\end{align*}
\begin{defi}\label{nonKahler locus singular}
Let $X$ be a compact normal space and $\alpha\in H^{1,1}_{BC}(X,\mathbb{R})$ be a big class. The non-K\"{a}hler locus of $\alpha$ is defined as
$$
E_{nK}(\alpha):=\pi\left(E_{nK}(\pi^*\alpha)\right)
$$
where $\pi:Y\to X$ is a resolution of singularities of $X$. The ample locus of $\alpha$ is the complement of the non-K\"{a}hler locus, that is, $\Amp(\alpha):=X\setminus E_{nK}(\alpha)$.
\end{defi}

\end{proof}
\begin{defi}[\cite{Bou2}]
Let $X$ be a compact K\"{a}hler manifold and $\alpha$ be a big class on $X$. Then the non-nef locus of $\alpha$ is defined as follows:
$$
E_{nn}(\alpha):=\{x\in X \mid \nu(\alpha,x)>0\},
$$
where $\nu(\alpha,x):=\nu(T_{\min},x)$ is the minimal multiplicity of $\alpha$.
\end{defi}

\begin{defi}[c.f. \cite{CT15}]
Let $X$ be a compact K\"{a}hler manifold and $\alpha$ be a nef and big class on $X$. The null locus of $\alpha$ is defined as follows:
$$
\Null(\alpha)=\bigcup_{\int_V\alpha^{\dim(V)}=0}V. 
$$
\end{defi}
The readers can consult with \cite{CT15} for the definition of the null locus of big classes.
Collins-Tosatti showed the following:
\begin{theo}[\cite{CT15}]\label{nonKahler-null}
Let $X$ be a compact K\"{a}hler manifold. For any nef and big class $\alpha$ on $X$, the non-K\"{a}hler locus coincides with the null locus:
$$
E_{nK}(\alpha)=\Null(\alpha).
$$
\end{theo}
We will use the following proposition later.
\begin{prop}\label{nef big singular exceptional}
Let $X$ be a compact normal space and $\alpha\in H^{1,1}_{BC}(X)$ be a nef and big class. Let $\pi:Y\to X$ be a resolution of singularities of $X$ and $D$ be a $\pi$-exceptional divisor with $\dim_X\pi(D)\le n-r$. Then, for any $\tau\in H^{r-k,r-k}(Y,\mathbb{R})$ with $\Supp(\tau)\subset D$ and $k>0$, we have
$$
(\pi^*\alpha)^{n-r+k}\cdot\tau=0.
$$
In particular 
$\Supp(D)\subset E_{nK}(\pi^*\alpha)$. 
\end{prop}
\begin{proof}
Let $\eta$ be a smooth representative of $\alpha$. Then $\pi^*\eta$ is a smooth representative of $\pi^*\alpha$. We also remark that $\eta^p$ is smooth $(p,p)$-form on $X$.
Let $\mathcal{N}_{Z}$ be the normal sheaf of $Z:=\pi(D)$ in $X$. Then the restriction $\pi|_D:D\to Z$ is isomorphic to the projective normal cone $\mathbb{P}(\mathcal{N}_Z)$ over $Z$. Let $U\subset Z_{\reg}$ be a Zariski open set such that $\mathcal{N}_Z|_U$ is locally free sheaf on $U$. Then $\mathbb{P}(\mathcal{N}_Z)|_U$ and $U$ are both smooth manifolds where wedge products commute with restrictions: 
\begin{align*}
\left.(\pi^*\eta)^{n-r+k}\right|_{ \mathbb{P}(\mathcal{N}_Z)|_U}
&=\left((\pi^*\eta)|_{\mathbb{P}(\mathcal{N}_Z)|_U}\right)^{n-r+k}\\
&=\left(\pi^*(\eta|_U)\right)^{n-r+k}\\
&=\pi^*\left((\eta|_U)^{n-r+k}\right)\\
&=0,
\end{align*}
the last equality follows from $\dim_X(Z)\le n-r$ and $k>0$. Let $\varphi_{\tau}\in H_{2(n-r+k)}(Y,\mathbb{R})$ be the Poincar\'{e} dual of $\tau$.  Then $\Supp(\varphi_{\tau})\subset \Supp(D)$ since $\tau$ is supported in $D$. Hence we have 
$$
\int_{\varphi_{\tau}}(\pi^*\eta)^{n-r+k}=\int_{\varphi_{\tau}}\left(\left.(\pi^*\eta)^{n-r+k}\right|_{\mathbb{P}(\mathcal{N}_Z)|_U}\right)=0.
$$
Therefore we obtain
\begin{align*}
(\pi^*\alpha)^{n-r+k}\cdot\tau=\int_{\varphi_{\tau}}(\pi^*\eta)^{n-r+k}
=0.
\end{align*}
If $\tau=[D]$, the above equation, together with Theorem \ref{nonKahler-null}, means that $\Supp(D)\subset E_{nK}(\pi^*\alpha)$.
\end{proof}

As the consequence of Theorem \ref{nonKahler-null}, we can prove the following well-known result.
\begin{prop}\label{semiample-bimero}
Let $X$ be a compact K\"{a}hler manifold and $\alpha$ be a big and semiample class on $X$. Then there exists a bimeromorphic morphism $\pi:X\to Y$ to a compact normal K\"{a}hler space $Y$ and $\omega\in H^{1,1}_{BC}(Y,\mathbb{R})$ a K\"{a}hler class on $Y$ such that $\pi^*\omega=\alpha$. Furthermore we have
$$
E_{nK}(\alpha)=\Null(\alpha)=\Exc(\pi).
$$
\end{prop}
\begin{proof}
Since $\alpha$ is semiample, there is a holomorphic surjection $\pi:X\to Y$ with connected fibres to a compact normal K\"{a}hler space $Y$ with a K\"{a}hler class $\omega$ on $Y$ such that $\alpha=\pi^*\omega$ by definition. We show this $\pi$ is bimeromorphic. Since $\alpha=\pi^*\omega$ and $\omega$ is K\"{a}hler, we can see the null locus of $\alpha$ coincides with the exceptional locus of $\pi$, that is,
$$
\Null(\alpha)=\bigcup_{\dim(f(V))<\dim(V)}V.
$$
Since any big and semiample class is big and nef, we have $\Null(\alpha)=E_{nK}(\alpha)$ by Theorem \ref{nonKahler-null}. Therefore, we obtain that the exceptional locus of $\pi$ coincides with the non-K\"{a}hler locus, in particular it is an proper analytic subvariety. Hence, for any $y\in Y\setminus \pi(E_{nK}(\alpha))$, its fibre $\pi^{-1}(y)$ is zero dimensional. Now we recall that any fibre of $\pi$ is connected. Thus $\pi^{-1}(y)$ consists of a point. Then we get the restriction $\pi|_{X\setminus E_{nK}(\alpha)}:X\setminus E_{nK}(\alpha)\to Y\setminus\pi(E_{nK}(\alpha))$ is bijective. Any bijective holomorphic map is biholomorphic.  Thus $\pi$ is bimeromorphic.
\end{proof}

\subsection{divisorial Zariski decomposition}\label{div.Zariski}
We have defined the positive product $\langle\alpha^p\rangle$ of a big class $\alpha$. If $p=1$, the positive product $\langle\alpha\rangle$ is described by the positive part of the divisorial Zariski decomposition:
\begin{defi}[\cite{Bou2}]
Let $X$ be a compact K\"{a}hler manifold and $\alpha$ be a big class on $X$. 
\begin{enumerate}
\item The divisorial Zariski decomposition is the decomposition as 
$$
\alpha=\langle\alpha\rangle+N(\alpha),
$$
here $N(\alpha)=\alpha-\langle\alpha\rangle$ the negative part of $\alpha$.
\item The divisorial Zariski decomposition is the Zariski decomposition if the positive part $\langle\alpha\rangle$ is nef.
\end{enumerate}
\end{defi}
\begin{rema}
The negative part $N(\alpha)$ of a big class $\alpha$ is described as
$$
N(\alpha)=\sum_{{\small{D:\hbox{irreducible divisor}}}}\nu(\alpha,D)[D],
$$
where $\nu(\alpha,D)=\nu(T_{\min},D)$ the Lelong number of $T_{\min}$ a closed positive $(1,1)$-current with minimal singularities in $\alpha$ (see \cite{Bou2}).
\end{rema}
The notion of the birational Zariski decomposition is also important. It is defined as follows:
\begin{defi}[c.f.\cite{BH14}]\label{bir.Zar.}
Let $X$ be a compact normal space and $\alpha\in H^{1,1}_{BC}(X)$ be big. We define that $\alpha$ admit a birational Zariski decomposition if there exists a resolution $\pi:Y\to X$ with $Y$ smooth K\"{a}hler such that $\pi^*\alpha$ admits the Zariski decomposition, that is, the positive part $\langle(\pi^*\alpha)\rangle$ is nef. 
\end{defi}

Let $f:Y\to X$ be a bimeromorphic morphism between compact K\"{a}hler manifolds. Let $\alpha$ be a big class on $X$ and $E$ be an effective $f$-exceptional divisor.
By \cite{DHY23}, we know that $N(f^*\alpha+E)=N(f^*\alpha)+E$. Therefore we have
\begin{lemm}\label{exc.div.prod}
Let $f:Y\to X$ be a bimeromorphic morphism between compact K\"{a}hler manifolds. Let $\alpha$ be a big class on $X$ and $E$ be an effective $f$-exceptional divisor. Then we have $\langle(f^*\alpha+E)^p\rangle=\langle (f^*\alpha)^p\rangle$.
\end{lemm}
We can see Lemma \ref{exc.div.prod} easily by the following lemma.
\begin{lemm}\label{positive-positive}
Let $X$ be a compact K\"{a}hler manifold and $\alpha$ be a big class on $X$. Then we have
$$
\langle\alpha^p\rangle=\langle\langle\alpha\rangle^p\rangle.
$$
\end{lemm}
\begin{proof}
Let $T_{\min}$ be a closed positive current with minimal singularities in $\alpha$. Then the difference between $\langle T_{\min}^p\rangle$ and $\langle \langle T_{\min}\rangle^p\rangle$ put mass only on an analytic subset. Since  both of them put no mass on analytic subsets, we obtain $\langle T_{\min}^p\rangle=\langle\langle T_{\min}\rangle^p\rangle$ and thus $\langle\alpha^p\rangle=\{\langle\langle T_{\min}\rangle^p\rangle\}\le \langle\langle\alpha\rangle^p\rangle$ (recall \cite[Theorem 1.16]{BEGZ}). 
For the inverse inequality, we recall the Siu decomposition (c.f. \cite{Bou2}) of $T_{\min}$, that is, $T_{\min}=\langle T_{\min}\rangle +N$ where $N=\sum\nu(\alpha,D)[D]$.
Let $R \in \langle\alpha\rangle$ be a closed positive current with minimal singularities.  Since $R+N \in \alpha$, we have $\{\langle(R+N)^p\rangle\}\le \langle \alpha^p\rangle$. Since $N$ is a sum of integral currents with positive coefficients, we have $\{\langle(R+N)^p\rangle\}=\{\langle R^p\rangle\}=\langle\langle\alpha\rangle^p\rangle$. 
\end{proof}

Let $f : Y \dashrightarrow X$ be a bimeromorphic map between compact normal varieties. We say $f$ is surjective in codimension 1 if the induced map $f_{*}: Z^{1}(Y)\rightarrow Z^{1}(X)$ is surjective. Here $Z^{1}(X)$ is the group of Weil divisors on $X$.
\begin{defi}
   A $bimeromorphic$ $contraction$ is a bimeromorphic map $f:Y\dashrightarrow X$ between compact normal analytic spaces which satisfies the following conditions:
   \begin{enumerate}
       \item $f: Y \dashrightarrow X$ is surjective in codimension 1.
       \item $f^{-1}:X \dashrightarrow Y$ does not contract divisors.
   \end{enumerate}
\end{defi}
\begin{defi}[\cite{DHY23}]\label{beta-negative def}
    Let $f:Y\dashrightarrow X$ be a bimeromorphic contraction between compact normal spaces. Let $\beta$ and $\alpha$ be pseudo-effective classes on $Y$ and $X$ respectively. Assume $\alpha=f_*\beta$. We say $f$ is \it{$\beta$-negative} if there exists a resolution $q:Z \rightarrow Y$ and $p:Z:\rightarrow X$ of $f$ and $E$ an effective $p$-exceptional divisor such that $q^*\beta=p^*\alpha+[E]$ holds and $\Supp(q_*E)$ coincides with the support of the $f$-exceptional divisors.
\end{defi}
\begin{exam}
Let $X$ be a smooth projective variety with $K_X$ big. Then, there exists a birational map $\pi: X \dashrightarrow X_{\can}$ to the canonical model $X_{\can}$ given by MMP \cite{BCHM}. Then this $\pi$ is $c_1(K_X)$-negative in the sense of Definition \ref{beta-negative def}.
\end{exam}
\begin{lemm}\label{surj-codim1-product}
    Let $f: Y \dashrightarrow X$ be a bimeromorphic map between compact normal spaces. Let $\beta$ and $\alpha$ be big classes on $Y$ and $X$, respectively. Suppose $f$ is $\beta$-negative, then
    \begin{equation}
        \langle (q^{*}\beta)^{k}\rangle=\langle (p^{*}\alpha)^{k}\rangle 
    \end{equation}
    holds for any $k=1,\ldots,n$, where $q:Z\to Y$ and $p: Z\to X$ are as in Definition \ref{beta-negative def}.
\end{lemm}
\begin{proof}
   Since $q^*\beta=p^*\alpha+E$ and $E$ are $p$-exceptional, this is a consequence of Lemma \ref{exc.div.prod}.

\end{proof}

Then we obtain the following important observation.
\begin{exam}
Let $X$ be a smooth projective variety of general type, that is, the canonical line bundle $K_X$ is big. Then, by \cite{BCHM}, $X$ has the canonical model, that is, a normal projective variety with canonical singularities $X_{\can}$ with $K_{X_{\can}}$ ample and a birational map $\pi:X\dashrightarrow X_{\can}$. Moreover, there exists a modification $p:Z\to X$ and $q:Z\to X_{\can}$ and effective $q$-exceptional divisor $E$ on $Z$ such that $p^*K_X=q^*K_{X_{\can}}+E$.
\begin{center}
\begin{tikzpicture}[auto]
\node (x) at (0,0) {$X$}; \node (x1) at (3,0) {$X_{\can}$}; \node (z) at (1.5,1.5) {$Z$};
\draw[->, dashed] (x) to node {$\pi$} (x1); \draw[->] (z) to node[swap] {$p$} (x); \draw[->] (z) to node {$q$} (x1);
\end{tikzpicture}
\end{center}
Then, by Lemma \ref{surj-codim1-product}, we have 
$$
\langle c_1(p^*K_X)^k\rangle=c_1(q^*K_{X_{\can}})^k.
$$
\end{exam}

\section{Assumption}\label{Assumption}
In this paper, we frequently need the following assumption:
\begin{assu}\label{assumption}
Let $\pi:Y\to X$ be a bimeromorphic morphism between compact K\"{a}hler manifolds. Let $\alpha$ be a big class on $X$ and $D$ be a $\pi$-exceptional divisor. Then we assume that the following holds:
$$
\langle(\pi^*\alpha)^{n-1}\rangle\cdot[D]=0.
$$
\end{assu}
This assumption is closely related to the following conjecture (c.f. \cite{Nystr19}).
\begin{conj}\label{conjecture}
Let $X$ be a compact K\"{a}hler manifold. Let $\alpha$ be a big class. Then the following {\rm{(1)}} and {\rm{(2)}} are conjectured to hold:
\begin{enumerate}
\item  Let $\gamma\in H^{1,1}(X,\mathbb{R})$ be any cohomology class. Then the folowing holds:
$$
\left.\frac{d}{dt}\right|_{t=0}\langle(\alpha+t\gamma)^n\rangle=n\gamma\cdot\langle\alpha^{n-1}\rangle.
$$

\item Let $\alpha=\langle\alpha\rangle+N(\alpha)$ be the divisorial Zariski decomposition of $\alpha$. Then the orthogonality relation $\langle\alpha^{n-1}\rangle\cdot N(\alpha)=0$ holds.
\end{enumerate}
\end{conj}
Witt Nystrom, in  \cite[Theorem D]{Nystr19}), proved this Conjecture \ref{conjecture} on projective manifolds and showed that the Assumption \ref{assumption} holds on projective manifolds:
\begin{theo}[\cite{Nystr19}]\label{nystr19}
If $X$ is smooth projective, then Conjecture \ref{conjecture} holds.
\end{theo}
On compact K\"{a}hler manifolds, Collins-Tosatti showed that 
\begin{theo}[\cite{CT15}]\label{thm of CT}
Let $X$ be a compact K\"{a}hler manifold and $\alpha$ be a big class on $X$. If $\alpha$ admits the Zariski decomposition, then Conjecture \ref{conjecture}(2) holds.
\end{theo}

\section{Slope stability with respect to Big Classes}\label{stability 0}
The aim of this section is to define the slope stability of reflexive sheaves with respect to a big class on a compact normal space.
\subsection{$\langle\alpha^{n-1}\rangle$-slope stability on smooth manifolds}
Let $X$ be a compact complex manifold of dimension $n$ and $\alpha$ be a big class on $X$.
\begin{defi}\label{stability defi}
Let $\mathcal{E}$ be a torsion free coherent sheaf on $X$. The $\langle\alpha^{n-1}\rangle$-degree and $\langle\alpha^{n-1}\rangle$-slope of $\mathcal{E}$ are defined as follows respectively,
\begin{equation}
\deg_{\alpha}(\mathcal{E}) := \int_{X} c_{1}(\det(\mathcal{E})) \wedge \frac{\langle \alpha^{n-1} \rangle}{(n-1)!}, \text{  } \mu_{\alpha}(\mathcal{E}) := \frac{\deg_{\alpha}(\mathcal{E})}{\rk(\mathcal{E})}.
\end{equation}
\end{defi} 
\begin{defi}\label{stability defi2}
Let $X$ be a compact complex manifold and $\alpha$ be a big class on $X$. A torsion free sheaf $\mathcal{E}$ on $X$ is $\langle\alpha^{n-1}\rangle$-slope stable if  the following holds : For any torsion free subsheaf $\mathcal{F} \subset \mathcal{E}$ of $0 < rk(\mathcal{F}) < rk(\mathcal{E})$ with torsion free quotient $\mathcal{E}/\mathcal{F}$, the following inequality of $\alpha$-slope holds,
\begin{equation}
\mu_{\alpha}(\mathcal{F}) < \mu_{\alpha}(\mathcal{E}).
\end{equation}
\end{defi}
\begin{rema}
Let $\mathcal{E}$ be a torsion free sheaf on a compact complex manifold of dimension $n$ and $\alpha$ be a big class on $X$. Then $\mathcal{E}$ is $\langle\alpha^{n-1}\rangle$-slope stable if and only if the reflexive hull $\mathcal{E}^{**}$ is $\langle\alpha^{n-1}\rangle$-slope stable \cite[Proposition 7.7]{Kob}.
\end{rema}
\begin{rema}
We can also define $\langle\alpha^{n-1}\rangle$-semistability and polystability in the usual way. We do not use these notions, in particular polystability, in this paper by assuming the irreducibility to reflexive sheaves. The same arguments in this paper work for polystable sheaves.
\end{rema}
\begin{exam}\label{stability invariance example}
Let $\pi:Y\to X$ be a blowup of compact K\"{a}hler manifold $X$. In this case we have $K_Y=\pi^*K_X+E$ where $E$ is effective $\pi$-exceptional. If $K_X$ is ample, then $K_Y$ is big and we have $\langle c_1(K_Y)^{n-1}\rangle=\pi^*c_1(K_X)^{n-1}$. Therefore, for any torsion free sheaf $\mathcal{E}$ on $Y$, we have that $\mathcal{E}$ is $\langle c_1(K_Y)^{n-1}\rangle$-stable if $\pi_{[*]}\mathcal{E}:=(\pi_{*}\mathcal{E})^{**}$ is $c_1(K_X)^{n-1}$-stable
\end{exam}
We will show the same type of bimeromorphic invariance of $\langle\alpha^{n-1}\rangle$-stability in more general setting (Theorem \ref{stability bimero}).

\subsection{$\langle\alpha^{n-1}\rangle$-slope stability on normal spaces}
In this section, we generalize the notion of $\langle\alpha^{n-1}\rangle$-stability of reflexive sheaves to singular setting. We frequently make use of Assumption \ref{assumption}.

\begin{defi}\label{stability normal defi}
Let $X$ be a compact normal space and $\alpha\in H^{1,1}_{BC}(X,\mathbb{R})$ be a big class on $X$. A torsion free sheaf $\mathcal{E}$ on $X$ is $\langle\alpha^{n-1}\rangle$-slope stable if for any resolution $\pi:\widehat{X}\to X$ of $X$ such that the reflexive pullback $\pi^{[*]}\mathcal{E}:=(\pi^{*}\mathcal{E})^{**}$ is $\langle(\pi^{*}\alpha)^{n-1}\rangle$-slope stable.
\end{defi}

In the following Lemma \ref{stability normal lemma}, we will prove that the above definition of $\langle\alpha^{n-1}\rangle$-stability is equivalent to that the reflexive pullback $\pi^{[*]}\mathcal{E}$ is $\langle(\pi^*\alpha)^{n-1}\rangle$-slope stable for some resolutions of $X$.  We remark that we need Assumption \ref{assumption} for Lemma \ref{stability normal lemma}. We also remark that if $\alpha$ is nef and big or a normal space is projective, we do not need Assumption \ref{assumption} (see Proposition \ref{nef big singular exceptional}, Theorem \ref{nystr19}).

\begin{lemm}\label{stability normal lemma}
    Let $X$ be a compact normal space of dimension $n$ with a big class $\alpha$ and $\mathcal{E}$ be a reflexive sheaf on X. Let $\pi_{i} : X_{i} \rightarrow X$ be resolutions of $X$ for $i=1,2$. Then $\pi_{1}^{[*]}\mathcal{E}$ is $\langle(\pi_{1}^{*}\alpha)^{n-1}\rangle$-stable if $\pi_{2}^{[*]}\mathcal{E}$ is $\langle(\pi_{2}^{*}\alpha)^{n-1}\rangle$-stable.
\end{lemm}
\begin{proof}
    Let $\mu_i:Y\rightarrow X_i$ ($i=1,2)$ be common resolutions of $\pi_i$. That is, $\mu_i$ are bimeromorphic morphisms satisfying $\pi_1\circ \mu_1=\pi_2\circ\mu_2$. 
    \begin{center}
    \begin{tikzpicture}[auto]
        \node (x) at (1.5,0) {$X$};
        \node (x1) at (0,1.5) {$X_1$};
        \node (x2) at (3,1.5) {$X_2$};
        \node (y) at (1.5,3) {$Y$};
        \draw[->] (x1) to node[swap] {$\pi_1$} (x);
        \draw[->] (x2) to node {$\pi_2$} (x);
        \draw[->] (y) to node[swap] {$\mu_1$} (x1);
        \draw[->] (y) to node {$\mu_2$} (x2);
    \end{tikzpicture}
    \end{center}
   
    It suffices to show that $\pi_i^{[*]}\mathcal{E}$ is $\langle(\pi_i^*\alpha)^{n-1}\rangle$-stable if $\mu_i^{[*]}(\pi_i^{*}\mathcal{E})$ is $\langle(\mu_i^*\pi_i^*\alpha)^{n-1}\rangle$-stable. We remark that $\mu_{i[*]}(\mu_i^{[*]}(\pi_i^*\mathcal{E}))=\pi_i^{[*]}\mathcal{E}$ and $\mu_1^{[*]}(\pi_1^*\mathcal{E})=\mu_2^{[*]}(\pi_2^*\mathcal{E})$. Furthermore $X_i$ and $Y$ are both smooth compact K\"{a}hler.
    Hence what we have to prove is reduced to the following claim:
    \begin{claim}\label{stability normal lemma claim}
        Let $f:Y\rightarrow X$ be a bimeromorphic map between compact K\"{a}hler manifolds. Let $\alpha$ be a big class on $X$. Let $\mathcal{E}$ and $\mathcal{F}$ be a reflexive sheaves on $X$ and $Y$ respectively. Suppose $\mathcal{F}\simeq f^{[*]}\mathcal{E}$ away from the $f$-exceptional locus. Then $\mathcal{E}$ is $\langle\alpha^{n-1}\rangle$-stable if $\mathcal{F}$ is $\langle (f^*\alpha)^{n-1}\rangle$-stable.
    \end{claim}
    \begin{proof}
        Assume $\mathcal{E}$ is $\langle\alpha^{n-1}\rangle$-stable. We show that $\mathcal{F}$ is $\langle(f^*\alpha)^{n-1}\rangle$-stable. There is a natural inclusion 
        $$
        f_*\mathcal{F}\hookrightarrow f_{[*]}\mathcal{F}\simeq\mathcal{E}.
        $$
        Let $\mathcal{G}\subsetneq \mathcal{F}$ be a nontrivial reflexive subsheaf of $\mathcal{F}$. Then the pushforward $f_*\mathcal{G}$ is a torsion free subsheaf of $\mathcal{E}$ via the natural inclusion
        $$
        \iota: f_*\mathcal{G}\hookrightarrow f_{[*]}\mathcal{F}\simeq\mathcal{E}.
        $$
        Denote by $\widetilde{\mathcal{G}}\subset \mathcal{E}$ the saturation sheaf of $f_*\mathcal{G}$ by $\iota$, that is, it is defined as $\widetilde{\mathcal{G}}:=\Ker(\mathcal{E}\rightarrow\mathcal{E}/{\rm{Im}}(\iota))$. Then $\widetilde{\mathcal{G}}$ is a reflexive subsheaf of $\mathcal{E}$ and it is nontrivial, since $1\le \rk(\mathcal{G})=\rk(\widetilde{\mathcal{G}})<\rk(\mathcal{F})=\rk(\mathcal{E})$. Since $\mathcal{E}$ is $\langle\alpha^{n-1}\rangle$-stable, we have $\mu_{\alpha}(\widetilde{\mathcal{G}})<\mu_{\alpha}(\mathcal{E})$. Next we show that 
        \begin{equation}\label{stability normal lemma claim eq2}
             \mu_{\alpha}(\widetilde{\mathcal{G}})=\mu_{f^*\alpha}(\mathcal{G}) \hbox{ and  }
            \mu_{\alpha}(\mathcal{E})=\mu_{f^*\alpha}(\mathcal{F}).
        \end{equation}
        Since $f^{[*]}\widetilde{\mathcal{G}}\simeq\mathcal{G}$ away from the $f$-exceptional locus, it suffices to show the second equation. 
        There is an $f$-exceptional divisor $D$ such that 
        $c_1(\mathcal{F})-c_1(f^{[*]}\mathcal{E})=D$ holds. We further remark that 
        $$
        c_1(f^{[*]}\mathcal{E})-f^*c_1(\mathcal{E})=c_1(\det f^{[*]}\mathcal{E})-c_1(f^*\det\mathcal{E})=\widetilde{D}
        $$
        is an $f$-exceptional divisor.
        We recall the Assumption \ref{assumption}. Then we obtain
        \begin{align}\label{stability normal lemma claim eq3}
            \mu_{f^*\alpha}(\mathcal{F})
            &=\frac{1}{\rk(\mathcal{F})}\int_{Y}c_1(\mathcal{F})\wedge\frac{\langle (f^*\alpha)^{n-1}\rangle}{(n-1)!} \notag \\
            &=\frac{1}{\rk(\mathcal{E})}\int_{Y}\left( c_1(f^{[*]}\mathcal{E})+D \right)\wedge\frac{\langle (f^*\alpha)^{n-1}\rangle}{(n-1)!} \notag\\
            &=\frac{1}{\rk(\mathcal{E})}\int_{Y}\left( c_1(f^*\mathcal{E})+\tilde{D} \right)\wedge \frac{\langle (f^*\alpha)^{n-1}\rangle}{(n-1)!} \notag\\
            &=\frac{1}{\rk(\mathcal{E})}\int_{X} c_1(\mathcal{E})\wedge \frac{\langle\alpha^{n-1}\rangle}{(n-1)!} \notag\\
            &=\mu_{\alpha}(\mathcal{E}).
        \end{align}
        Then we obtain (\ref{stability normal lemma claim eq2}) and thus $\mathcal{F}$ is $\langle(f^*\alpha)^{n-1}\rangle$-stable.

        Next we assume $\mathcal{F}$ is $\langle(f^*\alpha)^{n-1}\rangle$-stable. We now show $\mathcal{E}$ is $\langle\alpha^{n-1}\rangle$-stable. 
        The reflexive hull of the natural map $f^*f_*\mathcal{F}\rightarrow\mathcal{F}$
        induces $f^{[*]}f_{[*]}\mathcal{F}\rightarrow\mathcal{F}$. Since $f_{[*]}\mathcal{F}\simeq \mathcal{E}$, we obtain the natural sheaf morphism
        $$
        \iota: f^{[*]}\mathcal{E}\rightarrow\mathcal{F}.
        $$
        Since $f$ is bimeromorphic, this $\iota$ is isomorphic away from the $f$-exceptional locus. 
        Let $\mathcal{G}\subset \mathcal{E}$ be a nontrivial reflexive subsheaf. Then, the pullback of the inclusion $\mathcal{G}\hookrightarrow \mathcal{E}$ induces the following composition morphism:
        $$
        \eta:f^*\mathcal{G}\rightarrow f^*\mathcal{E}\rightarrow f^{[*]}\mathcal{E} \xrightarrow{\iota}\mathcal{F}.
        $$
        We denote by $\widetilde{f^*\mathcal{G}}\subset \mathcal{F}$ the saturation sheaf of the image sheaf $\eta(f^*\mathcal{G})\subset \mathcal{F}$. 
        Since $\widetilde{f^*\mathcal{G}}\simeq f^*\mathcal{G}$ away from the $f$-exceptional locus, there is an $f$-exceptional divisor $D$ such that $c_1(\widetilde{f^*\mathcal{G}})-c_1(f^*\mathcal{G})=D$ holds. Furthermore, there is an $f$-exceptional divisor $\widetilde{D}$ such that 
        $$
        c_1(f^*\mathcal{G})-f^*c_1(\mathcal{G})=c_1(\det f^*\mathcal{G})-c_1(f^*\det\mathcal{G})=\widetilde{D}.
        $$
        Then by the same calculation with (\ref{stability normal lemma claim eq3}), we obtain
        $$
        \mu_{\alpha}(\mathcal{G})=\mu_{f^*\alpha}(\widetilde{f^*\mathcal{G}})<\mu_{f^*\alpha}(\mathcal{F})=\mu_{\alpha}(\mathcal{E}).
        $$
    \end{proof}
    Then we end the proof of Lemma \ref{stability normal lemma}.
\end{proof}

\subsection{bimeromorphic invariance of $\langle\alpha^{n-1}\rangle$-stability}  

The notion of $\langle\alpha^{n-1}\rangle$-stability is invariant under a suitable bimeromorphic map. We need the Assumption \ref{assumption} for the following theorem.
\begin{theo}\label{stability bimero}
    Let $(X, \mathcal{E},\alpha)$ and $(Y, \mathcal{F}, \beta)$ be triples consisting of a compact normal space, a reflexive sheaf and a big class. Let $f:Y\dashrightarrow X$ be a $\beta$-negative bimeromorphic map. Suppose 
        $\mathcal{F}\simeq f^{[*]}\mathcal{E}$ away from the $f$-exceptional locus.
    Then, $\mathcal{F}$ is $\langle\beta^{n-1}\rangle$-stable if and only if $\mathcal{E}$ is $\langle\alpha^{n-1}\rangle$-stable.
\end{theo}
\begin{proof}
    Let $p:Z \to Y$ and $q:Z\to X$ be bimeromorphic morphisms from a compact K\"{a}hler manifold $Z$ so that $p^*\beta-q^*\alpha=E$ is an effective $q$-exceptional divisor such that $p(E)$ is $f$-exceptional.
    \begin{center}
    \begin{tikzpicture}[auto]
    \node (y) at (0,0) {$Y$};
    \node (x) at (2.6,0) {$X$};
    \node (z) at (1.3,1.3) {$Z$};
    \draw[->] (z) to node[swap] {$p$} (y);
    \draw[->] (z) to node {$q$} (x);
    \draw[->, dashed] (y) to node {$f$} (x);
    \end{tikzpicture}
    \end{center}
    By assumption, we have
    \begin{itemize}
        \item[(a)] $\langle(p^*\beta)^{n-1}\rangle=\langle(q^*\alpha)^{n-1}\rangle$ and
        \item[(b)] $p^{[*]}\mathcal{F}\simeq q^{[*]}\mathcal{E}$ away from the $q$-exceptional locus. 
    \end{itemize}
    Suppose that $\mathcal{E}$ is $\langle\alpha^{n-1}\rangle$-stable. Then, by Lemma \ref{stability normal lemma}, we have that $q^{[*]}\mathcal{E}$ is $\langle(q^*\alpha)^{n-1}\rangle$-stable. From (a), (b) above and Claim \ref{stability normal lemma claim}, we obtain $p^{[*]}\mathcal{F}$ is $\langle(p^*\beta)^{n-1}\rangle$-stable. By Lemma \ref{stability normal lemma} again, it means that $\mathcal{F}$ is $\langle\beta^{n-1}\rangle$-stable.

    Conversely, we assume that $\mathcal{F}$ is $\langle\beta^{n-1}\rangle$-stable. Then $p^{[*]}\mathcal{F}$ is $\langle(p^*\beta)^{n-1}\rangle$-stable. By (b) above, we can only say that $p^{[*]}\mathcal{F}$ coincides with $q^{[*]}\mathcal{E}$ only out of $q$-exceptional locus which is larger than the $p$-exceptional locus.  But by (b), we can conclude that $q^{[*]}\mathcal{E}$ is $\langle(q^*\alpha)^{n-1}\rangle$-stable by the same way with Claim \ref{stability normal lemma claim}. It implies that $\mathcal{E}$ is $\langle\alpha^{n-1}\rangle$-stable.
\end{proof}

A normal projective variety $X$ is $\mathbb{Q}$-Gorenstein if the canonical divisor $K_{X}$ is $\mathbb{Q}$-Cartier.
In this case, we define $c_1(K_{X}):=\frac{1}{r}c_1(rK_{X})$ where $r$ is an integer such that $rK_{X}$ is a line bundle. 
We say that a normal projective variety $X$ is of general type if there is a resolution $\pi : \widehat{X}\rightarrow X$ so that $K_{\widehat{X}}$ is big. 
\begin{lemm}
    Let $X$ be a normal $\mathbb{Q}$-Gorenstein projective variety. If $X$ is of general type, then $c_1(K_{X})$ is big in the sense of Definition \ref{big normal}. That is, there is a resolution $\pi:\widehat{X}\rightarrow X$ such that $\pi^*c_1(K_X)$ is big. 
\end{lemm}
\begin{proof}
    Let $r\in \mathbb{Z}_{>0}$ be an integer such that $rK_X$ is locally free. Since $X$ is of general type, there is a resolution $\pi:\widehat{X}\rightarrow X$ such that $K_{\widehat{X}}$ is big. Since $rK_{\widehat{X}}-\pi^*(rK_X)=E$ where $E$ is a (not necessarily effective) $\pi$-exceptional divisor, we have $\pi_*(rK_{\widehat{X}})=rK_X$, which is Cartier. Since a line bundle $rK_{\widehat{X}}$ is big, it contains a K\"{a}hler current $T$. Its push-forward $\pi_*T$ is a K\"{a}hler current contained in $\pi_*(rK_{\widehat{X}})=rK_X$. Since the pull-back $\pi^*(\pi_*T)\in \pi^*c_1(rK_X)$ satisfies $\pi^*(\pi_*T)\ge \pi^*\omega$, its volume $\langle\pi^*T^{n}\rangle$ is positive. Thus $\pi^*c_1(rK_X)$ is big.
\end{proof}
\begin{exam}\label{cor-stability-tangent-big}
    Let $X$ be a normal projective variety of general type with canonical singularities. Then the tangent sheaf $\mathcal{T}_{X}$ of $X$ is $\langle c_1(K_{X})^{n-1}\rangle$-slope polystable.
\end{exam}
 In fact, let $\mu:\widehat{X}\rightarrow X$ be a resolution of $X$. Since the normal variety $X$ is of general type, $\widehat{X}$ is of general type, by definition. In this case the minimal model program (MMP) starting at $\widehat{X}$ terminates  and thus there is a birational map $f:\widehat{X}\dashrightarrow X_{can}$ given by MMP (c.f.\cite{BCHM}). This $f$ is $K_{\widehat{X}}$-negative. Here $X_{can}$ is the canonical model of $\widehat{X}$. The tangent sheaf $\mathcal{T}_{X_{can}}$ of $X_{can}$ is $K_{X_{can}}$ by \cite{Gue}. By Theorem \ref{stability bimero}, we obtain the $\langle c_1(K_{\widehat{X}})^{n-1}\rangle$-stability of $\mathcal{T}_{\widehat{X}}$. Since $X$ has canonical singularities, we have $K_{\widehat{X}}= \mu^*K_{X}+E$ where $E$ is an effective $\mu$-exceptional divisor. Hence $\langle c_1(K_{\widehat{X}})^{k}\rangle=\langle c_1(\mu^*K_{X})^k\rangle$ and thus we obtain $\langle c_1(K_{X})^{n-1}\rangle$-stability of $\mathcal{T}_{X}$.

\section{$T$-Hermitian-Einstein metrics}\label{THE 0}
In this section, we define the notion of $T$-Hermitian-Einstein metric for $T$ a closed positive $(1,1)$-current (Definition \ref{THE defi}, Definition \ref{THE normal}) and we will see the bimeromorphic invariance of the existence of $T$-Hermitian-Einstein metrics.

Let $X$ be a compact normal space and $\mathcal{E}$ be a coherent sheaf on $X$. We denote by $\Sing(\mathcal{E})$ the non locally free locus of $\mathcal{E}$. This is a proper analytic subset of $X$ \cite{Kob}.
\subsection{$T$-Hermitian-Einstein metrics on smooth manifolds}
\begin{defi}\label{THE defi}
Let $X$ be a compact complex manifold of dimension $n$, $\alpha$ be a big class on $X$ and $\mathcal{E}$ be a reflexive sheaf on $X$. 
Denote as $\Omega:=\Amp(\alpha)\setminus\Sing(\mathcal{E})$.
Let $T\in\alpha$ be a closed positive $(1,1)$-current which is smooth K\"{a}hler on $\Amp(\alpha)$. A $T$-Hermitian-Einstein metric in $\mathcal{E}$ is a smooth hermitian metric $h$ in $\mathcal{E}|_{\Omega}$ whose curvature tensor $F_h$ satisfies
\begin{enumerate}
\item $\sqrt{-1}\Lambda_{T}F_h=\lambda\cdot \Id_{\mathcal{E}}$ on $\Omega$ for some constant $\lambda \in \mathbb{R}$,
\item $\int_{\Omega}|F_h|_T^2T^n<\infty$ and
\item the constant $\lambda$ in (1), called as the $\langle\alpha^{n-1}\rangle$-$Hermitian$-$Einstein$ $constant$, satisfies
$$
\lambda=\frac{1}{\int_X\langle\alpha^n\rangle/n!}\int_Xc_1(\det(\mathcal{E}))\wedge\frac{\langle\alpha^{n-1}\rangle}{(n-1)!}.
$$
\end{enumerate}
\end{defi}
\begin{defi}
Let $X$ be a compact complex manifold. Let $\alpha$ be a big class on $X$ and $\mathcal{E}$ be a reflexive sheaf on $X$. Then we say that $\mathcal{E}$ admits a $T$-Hermitian-Einstein metric if there exists a closed positive $(1,1)$-current $T\in\alpha$ on $X$ and a smooth hermitian metric $h$ in $\left.\mathcal{E}\right|_{\Amp(\alpha)\setminus\Sing(\mathcal{E})}$ such that $h$ is a $T$-Hermitian-Einstein metric in $\mathcal{E}$.
\end{defi}
\begin{exam}[compare with Example \ref{stability invariance example}]\label{HE blowup}
Let $\pi: Y\to X$ be a blow up of a compact complex manifold $X$ with $K_X$ ample and $\mathcal{E}$ be a reflexive sheaf on $Y$. Then $K_Y=\pi^*K_X+E$ is big. Recall that $\langle c_1(K_Y)^{k}\rangle=\pi^*c_1(K_X)^{k}$. 
Hence  the $\langle c_1(K_Y)^{n-1}\rangle$-HE constant of $\mathcal{E}$ equals to the $c_1(K_X)^{n-1}$-HE constant of the reflexive push forward $\pi_{[*]}\mathcal{E}:=(\pi_*\mathcal{E})^{**}$. Let $T:=\pi^*\omega+[E]$ be a closed positive $(1,1)$-current where $\omega$ is a K\"{a}hler metric in $c_1(K_X)$. Then $T$-HE metric is nothing but the pull back of the admissible $\omega$-HE metric in $\pi_{[*]}\mathcal{E}$.
\end{exam}

\subsection{$T$-Hermitian-Einstein metrics on normal spaces}
We define the notion of $T$-HE metric in reflexive sheaves on compact normal spaces. 
\begin{defi}\label{THE normal}
Let $X$ be a compact normal space. We fix a resolution $\pi: Y\to X$ of $X$. Let
$\alpha\in H^{1,1}_{BC}(X,\mathbb{R})$ be a big class on $X$ and
$\mathcal{E}$ be a reflexive sheaf on $X$. 
We say that $\mathcal{E}$ admits a $T$-Hermitian-Einstein metric if there exists
\begin{itemize}
\item a resolution of singularities $\pi:Y\to X$ of $X$ and
\item a closed positive $(1,1)$-current $T\in\pi^*\alpha$ on $Y$
\end{itemize}
such that $\pi^{[*]}\mathcal{E}$ admits a $T$-HE metric.

\end{defi}
\begin{lemm}
Let $X$ be a compact normal space, $\alpha$ be a big class on $X$ and $\mathcal{E}$ be a reflexive sheaf on $X$. Let $\pi_i:X_i\to X$ be resolutions of singularities of $X$ for $i = 1, 2$.
Then $\pi_1^{[*]}\mathcal{E}$ admits a $T_1$-HE metric if and only if $\pi_2^{[*]}\mathcal{E}$ admits a $T_2$-HE metric where each $T_i \in \pi^*_i\alpha$ is a closed positive $(1,1)$-current which is smooth K\"{a}hler on $\Amp(\pi^*_i\alpha)$ for $i=1, 2$.
\end{lemm}
\begin{proof}
We first remark that, by Assumption \ref{assumption}, the $\langle(\pi_1^*\alpha)^{n-1}\rangle$-HE constant of $\pi_1^{[*]}\mathcal{E}$ coincides with the $\langle(\pi_2^*\alpha)^{n-1}\rangle$-HE constant of $\pi_2^{[*]}\mathcal{E}$.

Let us choose a common resolution of $\pi_1$ and $\pi_2$.
\begin{center}
\begin{tikzpicture}[auto]
\node (x1) at (0,1) {$X_1$}; \node (x2) at (1.8,1) {$X_2$}; \node (x) at (0.9,0) {$X$}; \node (y) at (0.9, 2) {$Y$};
\draw[->] (x1) to node[swap] {$\pi_1$} (x); \draw[->] (x2) to node {$\pi_2$} (x); \draw[->] (y) to node[swap] {$p_1$} (x1); \draw[->] (y) to node {$p_2$} (x2);
\end{tikzpicture}
\end{center}
Let us consider a reflexive sheaf $p_i^{[*]}(\pi_i^*\mathcal{E})$ on $Y$. We can see that $p_i^{[*]}(\pi_i^*\mathcal{E})$ differs from $p_i^{[*]}(\pi_i^{[*]}\mathcal{E})$ only on the $p_i$-exceptional locus $\Exc(p_i)$. On the other hand, we know that $\Amp(p_i^*(\pi^*\alpha))=p_i^{-1}\left( \Amp(\pi_i^*\alpha)\right) \setminus \Exc(p_i)$ by \cite[Lemma 4.3]{CT15}. Therefore we obtain
$$
\Amp(p_i^*(\pi_i^*\alpha))\setminus \Sing(p_i^{[*]}(\pi_i^*\mathcal{E}))
= p_i^{-1}\left( \Amp(\pi_i^*\alpha)\setminus \Sing(\pi_i^{[*]}\mathcal{E})\right).
$$
We further remark that $p_1^{[*]}(\pi_1^*\mathcal{E})=p_2^{[*]}(\pi_2^*\mathcal{E})$ and it differs from $p_2^{[*]}\pi^{[*]}_2\mathcal{E}$ only on the $p_2$-exceptional locus.
Therefore, if $h$ is a $T$-HE metric in $\pi_1^{[*]}\mathcal{E}$ for some $T\in \pi^*_1\alpha$, then $p_{2 *}(p_1^*h)$ is a $p_{2 *}(p_1^*T)$-HE metric in $\pi_2^{[*]}\mathcal{E}$.
\end{proof}

\subsection{bimeromorphic invariance of $T$-Hermitian-Einstein metrics}
We need Assumption \ref{assumption} for the following theorem.
\begin{theo}\label{birational HE}
Let $(Y,\beta,\mathcal{F})$ and $(X,\alpha,\mathcal{E})$ be triples consisting of a compact normal space, a big class and a reflexive sheaf. Let $\pi:Y\dashrightarrow X$ be a bimeromorphic map. Suppose
\begin{itemize}
\item $\pi: Y \dashrightarrow X$ is $\beta$-negative in the sense of Definition \ref{beta-negative def} and
\item $\mathcal{F}\simeq \pi^{[*]}\mathcal{E}$ away from the $\pi$-exceptional locus.
\end{itemize}
Then, $\mathcal{E}$ admits a $T$-HE metric $h_T$ if $\mathcal{F}$ admits a $\pi^*T$-HE metric $h_{\pi^*T}$. 
More precisely, let $p: Z \to Y$ and $q : Z \to X$ be resolutions of indeterminacy of $\pi$ and $E$ be the effective $\mathbb{R}$-divisor so that $p^*\beta=q^*\alpha+[E]$ holds. Let $T \in q^*\alpha$ be a closed positive $(1,1)$-current which is smooth K\"{a}hler on $\Amp(q^*\alpha)$. Then $T$-HE metric $h_T$ in $q^{[*]}\mathcal{E}$ gives a $(T + [E])$-HE metric in $p^{[*]}\mathcal{F}$.

\end{theo}
\begin{proof}
Let us choose resolutions of indeterminacy of $\pi$, denoted by $q:Z\to Y$ and $p:Z\to X$, so that $Z$ is smooth K\"{a}hler and $q^*\beta-p^*\alpha=E$ is effective $p$-exceptional divisor. 
 \begin{center}
    \begin{tikzpicture}[auto]
    \node (y) at (0,0) {$Y$};
    \node (x) at (2.6,0) {$X$};
    \node (z) at (1.3,1.3) {$Z$};
    \draw[->] (z) to node[swap] {$p$} (y);
    \draw[->] (z) to node {$q$} (x);
    \draw[->, dashed] (y) to node {$\pi$} (x);
    \end{tikzpicture}
    \end{center}
Then we have
$$
\Amp(q^*\beta)=\Amp(p^*\alpha)\setminus E.
$$
We denote by $\Exc(p)$ and $\Exc(q)$ the exceptional sets of $p$ and $q$ respectively. Then we have
$$
\Exc(p)=\Exc(q)\cup E.
$$
Since $\mathcal{F}\simeq \pi^{[*]}\mathcal{E}$ away from the $\pi$-exceptional set, we have $q^{[*]}\mathcal{F}\simeq p^{[*]}\mathcal{E}$ on $Z \setminus (\Exc(q)\cup E)$. Therefore we obtain
$$
\Amp(q^*\beta)\setminus\Sing(q^{[*]}\mathcal{F})=\Amp(p^*\alpha)\setminus\Sing(p^{[*]}\mathcal{E}).
$$
Furthermore, by (\ref{stability normal lemma claim eq2}), the $\langle(p^*\alpha)^{n-1}\rangle$-HE constant of $p^{[*]}\mathcal{E}$ coincides with the $\langle(q^*\beta)^{n-1}\rangle$-HE constant of $q^{[*]}\mathcal{F}$.  Let $T\in p^*\alpha$ be a closed positive $(1,1)$-current which is smooth K\"{a}hler on $\Amp(p^*\alpha)$. Then a $T$-HE metric in $p^{[*]}\mathcal{E}$ gives a $(T+ [E])$-HE metric in $q^{[*]}\mathcal{F}$. 
\end{proof}

\section{Kobayashi-Hitchin correspondence}\label{KH corre}

In this section, we prove the Kobayashi-Hitchin correspondence. The proof is an application of the bimeromorphic invariance of $\langle\alpha^{n-1}\rangle$-slope stability (Theorem \ref{stability bimero}) and the existence of $T$-Hermitian-Einstein metrics (Theorem \ref{birational HE}). As a corollary, we obtain a complete proof the Kobayashi-Hitchin correspondence on projective variety of general type. The existence of the canonical model by \cite{BCHM} plays an essential role.

For the following theorem, we ``do not'' need Assumption \ref{assumption} since a big class in the statement admits a birational Zariski decomposition. 
\begin{theo}\label{KH corr}
Let $X$ be a compact normal space and $\alpha\in H^{1,1}_{BC}(X,\mathbb{R})$ be a big class on $X$. Let $\mathcal{E}$ be a reflexive sheaf on $X$. Suppose $\alpha$ admits a birational Zariski decomposition whose positive part is big and semiample. Then $\mathcal{E}$ is $\langle\alpha^{n-1}\rangle$-stable if $\mathcal{E}$ admits a $T$-HE metric.

A closed positive $(1,1)$-current $T$ is given as follows: Let $\mu: Z \to X$ be a resolution of $X$ such that the positive part of the divisorial Zariski decomposition $\mu^*\alpha=\langle\mu^*\alpha\rangle + [E]$ is semiample and big where $[E]$ is effective $\mathbb{R}$-divisor. Let $\omega'\in \langle\mu^*\alpha\rangle$ be a pullback of some K\"{a}hler metric. Then $T=\omega'+[E]$ is a closed positive $(1,1)$-current in the above statement. 
\end{theo}\label{KH kahler}
\begin{proof}
Let $\mu:Z\to X$ be a modification so that $Z$ is smooth K\"{a}hler and the divisorial Zariski decomposition $\mu^*\alpha=\langle\mu^*\alpha\rangle+D$ gives big and semiample positive part. 
Let $\pi:Z\to Y$ be a bimeromorphic morphism to a compact normal K\"{a}hler space $Y$ with a K\"{a}hler class $\omega$ on $Y$ such that $\langle\mu^*\alpha\rangle=\pi^*\omega$. 
\begin{center}
\begin{tikzpicture}[auto]
\node (z) at (1,1) {$Z$}; \node (x) at (0,0) {$X$}; \node (y) at (2,0) {$Y$};
\draw[->] (z) to node[swap] {$\mu$} (x); \draw[->] (z) to node {$\pi$} (y);
\end{tikzpicture}
\end{center}
Suppose $\mathcal{E}$ is $\langle\alpha^{n-1}\rangle$-stable. Then, by Lemma \ref{stability normal lemma} and Theorem \ref{thm of CT}, $\mu^{[*]}\mathcal{E}$ is $\langle(\mu^*\alpha)^{n-1}\rangle=\pi^*\omega^{n-1}$-stable. Therefore, by Theorem \ref{stability bimero}, $\pi_{[*]}(\mu^{[*]}\mathcal{E})$ is $\omega^{n-1}$-stable. The result in \cite{Chen} ensures that $\pi_{[*]}(\mu^{[*]}\mathcal{E})$ admits the admissible $\omega$-HE metric $h$. By Proposition \ref{semiample-bimero}, we can see $D\subset E_{nK}(\mu^*\alpha)=\Exc(\pi)$. Let $T:=(\pi^*\omega+[D])$ be a closed positive $(1,1)$-current in $\mu^*\alpha$. Then, by Theorem \ref{birational HE}, its pullback $\pi^*h$ gives the $T$-HE metric in $\mu^{[*]}\mathcal{E}$. 
Next we assume that $\mu^{[*]}\mathcal{E}$ admits a $T:=(\pi^*\omega+[D])$-HE metric $h$. We remark that $\pi_*T=\omega$. Since $D$ is $\pi$-exceptional, the pushforward $\pi_*h$ is exactly the $\omega$-admissible HE metric in $\pi_{[*]}\mu^{[*]}\mathcal{E}$. Hence, again by the result in \cite{Chen}, we know $\pi_{[*]}\mu^{[*]}\mathcal{E}$ is $\omega^{n-1}$-stable. Hence $\mu^{[*]}\mathcal{E}$ is $\langle\mu^*\alpha^{n-1}\rangle=\pi^*\omega^{n-1}$-stable. It means that $\mathcal{E}$ is $\langle\alpha^{n-1}\rangle$-stable by Lemma \ref{stability normal lemma}.
\end{proof}

\begin{corr}\label{KH 1'}
Let $X$ be a normal projective variety with log terminal singularities, where $K_X$ is $\mathbb{R}$-Cartier. Let $\mathcal{E}$ be a reflexive sheaf on $X$. If $K_X$ is big, then $\mathcal{E}$ is $\langle c_1(K_X)^{n-1}\rangle$-stable if and only if $\mathcal{E}$ admits a $T$-HE metric.
\end{corr}
\begin{proof}
By \cite{BCHM}, there exists the log canonical model of $X$. That is, there is $X_{\can}$ a normal projective variety with log canonical singularities where $K_{X_{\can}}$ is ample, and a birational contraction $\varphi:X\dashrightarrow X_{\can}$ which is $K_X$-negative. Therefore, there exists resolutions of indeterminacy of $\pi$, denoted by $p:Y\to X$ and $q:Y\to X_{\can}$, such that $p^*K_X-q^*K_{X_{\can}}=E$ is effective $q$-exceptional. The decomposition $p^*K_X=q^*K_{X_{\can}}+E$ gives the birational Zariski decomposition of $K_X$ with big and semiample positive part $\langle p^*K_X\rangle=q^*K_{X_{\can}}$. Therefore, by Theorem \ref{KH corr}, $\mathcal{E}$ is $\langle c_1(K_X)^{n-1}\rangle$-stable if $p^{[*]}\mathcal{E}$ admits a $T$-HE metric where $T=q^*\omega_{\can}+[E]$ for any K\"{a}hler metric $\omega_{\can}\in c_1(K_{X_{\can}})$.
\end{proof}

By \cite{Gue}, the tangent sheaf $\mathcal{T}_{X_{\can}}$ is $c_1(K_{X_{\can}})^{n-1}$-polystable. Hence we obtain the following.
\begin{exam}
Let $X$ be a normal projective variety with log terminal singularities where $K_X$ is $\mathbb{R}$-Cartier. If $K_X$ is big, then the tangent sheaf $\mathcal{T}_{X}$, the cotangent sheaf $\Omega_{X}^{[1]}$, their tensor products and wedge products are all $K_{X}$-slope polystable, and thus admit $T$-HE metrics.
\end{exam}

\section{Bogomolov-Gieseker Inequality for big and nef classes}\label{BG ineq}
In this section, we prove the Bogomolov-Gieseker inequality for nef and big classes on compact normal spaces (Corollary \ref{BG nef big 2}). We also obtain the characterization of the equality in a special setting (Theorem \ref{BG equality}).  In the proof, the ``openness'' of  slope stability with respect to big classes plays an essential role (see \S \ref{openness stab}).
In this section, we ``do not need'' Assumption \ref{assumption}.
\begin{lemm}[\cite{Chen}, Lemma 2.3]\label{split of class}
Let $X$ be a compact K\"{a}hler manifold and $V$ be a submanifold of $\codim(V)\ge p$. Let $\eta\in H^{n-p,n-p}(X,\mathbb{R})$ be a cohomology class satisfying $\eta|_V=0$. Then, for deformation retracts $N_1\Subset N_2$ of $V$, there is a closed $(n-p,n-p)$-form $\Phi$ and $(2n-2p-1)$ form $\Psi$ on $X$ such that
\begin{itemize}
\item $\Supp(\Phi)\subset X\setminus\overline{N_1}$,
\item $\Supp(\Psi)\subset N_2$ and
\item $\eta=\Phi+d\Psi$ as a smooth differential form.
\end{itemize}
\end{lemm}
\begin{proof}
Although this lemma is proven in \cite{Chen}, we note the proof for the readers.
Let $N_1\Subset N_2$ be two deformation retracts of $V$. We have $H^{n-p,n-p}(V)\simeq H^{n-p.n-p}(N_2)$ and thus $\eta|_{N_2}=0$ as a singular cohomologies. Now $X$ and $V$ are smooth  and thus we can choose $N_i$ as smooth submanifold. Thus we have $\eta|_{N_2}=0$ as a de-Rham cohomology. Hence there exists a smooth $(2n-2p-1)$ form $\Psi'$ on $N_2$ such that $\eta|_{N_2}=d\Psi'$ as a smooth form. Let $\rho:X\to \mathbb{R}_{\ge 0}$ be a bump function which satisfies $\equiv 1$ on $N_1$ and $\equiv 0$ on $X\setminus N_2$. Then $\Psi:=\rho\Psi'$ and $\Phi:=\eta-d\Psi$ are what we wanted. 
\end{proof}

\subsection{openness of stability}\label{openness stab}
\begin{lemm}\label{deg bound}
Let $X$ be a compact K\"{a}hler manifold and $\mathcal{E}$ be a reflexive sheaf on $X$. Let $\gamma_{\varepsilon}\in H^{n-1,n-1}(X,\mathbb{R})$ be a sequence of cohomology classes each of which is represented by a positive $(n-1,n-1)$-current. Suppose $(\gamma_{\varepsilon})_{\varepsilon}$ is contained in a bounded subset in $H^{n-1,n-1}(X, \mathbb{R})$. 
Then, there is a constant $C > 0$ such that the following inequality holds for any reflexive subsheaf $\mathcal{F}$ of ${\mathcal{E}}$ and any $0 \leq \varepsilon \ll1$, 
\begin{equation}
deg(\mathcal{F}, \gamma_{\varepsilon}):=\int_Xc_1(\det\mathcal{F})\wedge \frac{\gamma_{\varepsilon}}{(n-1)!} \leq C.
\end{equation}
If $\gamma_{\varepsilon}\rightarrow 0$ in $\varepsilon\rightarrow0$, then for any $N \in \mathbb{Z}_{>0}$, there exists $\varepsilon_{0}>0$ such that 
\begin{equation}
    deg(\mathcal{F}, \gamma_{\varepsilon})<\frac{1}{N}
\end{equation}
holds for any reflexive subsheaf $\mathcal{F}\subset {\mathcal{E}}$ and $0<\varepsilon<\varepsilon_{0}$.
\end{lemm}
\begin{proof}
Since $\deg(\mathcal{F},\gamma_{\varepsilon})=\deg(\pi^{[*]}\mathcal{F},\pi^*\gamma_{\varepsilon})$ for any resolution $\pi$, we can assume that $\mathcal{E}$ is locally free.
Let $h_{0}$ be a smooth hermitian metric in ${\mathcal{E}}$ and $p : {\mathcal{E}} \rightarrow {\mathcal{E}}$ be the $h_{0}$-orthogonal projection to $\mathcal{F}$ defined on the Zariski open set where $\mathcal{F}$ is locally free. Let $\nu : \widehat{X} \rightarrow {X}$ be a resolution so that the saturation sheaf of $\nu^*\mathcal{F}$ in $\nu^*\mathcal{E}$, denoted by $\widehat{\mathcal{F}}$, is a locally free subsheaf of a vector bundle $\nu^{*}{\mathcal{E}}$.  Let $\widehat{p}$ be the $\nu^{*}h_{0}$-orthogonal projection to $\widehat{\mathcal{F}}$. This projection $\widehat{p}$ is smooth and $\nu^{*}p = \widehat{p}$ away from the $\nu$-exceptional divisor. The equation
\begin{equation}
c_{1}(\widehat{\mathcal{F}}) = \nu^{*}c_{1}(\mathcal{F}) + c_{1}(D)
\end{equation}
holds where $D$ is the $\nu$-exceptional divisor. Since $\codim(\nu(D))\geq 2$, we have $(\nu^{*}\gamma_{\varepsilon})|_{D} = 0 \in H^{n-1, n-1}(D, \mathbb{R})$. Then we apply Lemma \ref{split of class} for $\nu^{*}\gamma_{\varepsilon}$. We recall that $c_{1}(D)$ equals $0$ away from $D$. Then we have 
\begin{equation}
\int_{X} c_{1}(D) \wedge \nu^{*}\gamma_{\varepsilon} = 0.
\end{equation}
Then, we can calculate as follows and thus obtain the first assertion.
\begin{align}\label{deg bound eq}
deg(\mathcal{F}, \gamma_{\varepsilon}) 
&= \int_{{X}}c_{1}(\mathcal{F}) \wedge \frac{\gamma_{\varepsilon}}{(n-1)!} \notag \\
&= \int_{\widehat{X}} c_{1}(\widehat{\mathcal{F}}) \wedge \frac{\nu^{*}\gamma_{\varepsilon}}{(n-1)!} \notag \\
&= \int_{\widehat{X}} Tr (\widehat{p}\cdot F_{\nu^{*}h_{0}}\cdot \widehat{p} + \bar{\partial}\widehat{p} \wedge \partial_{h_{0}}\widehat{p}) \wedge \frac{\nu^{*}\gamma_{\varepsilon}}{(n-1)!} \notag \\
&\leq  \int_{{X}\setminus \nu(D)} Tr (p\cdot F_{h_{0}}\cdot p) \wedge \frac{\gamma_{\varepsilon}}{(n-1)!} \notag \\
&\leq \|F_{h_{0}}\|_{L^{\infty}}rk(\mathcal{F}) \int_{\widehat{X}} \omega\wedge\frac{\gamma_{\varepsilon}}{(n-1)!} \notag \\
&\leq C .
\end{align}
If $\gamma_{\varepsilon}\rightarrow 0$, we can easily see the second assertion from the above inequality. We end the proof.
\end{proof}
\begin{lemm}\label{max-slope}
    Let $X$ be a compact K\"{a}hler manifold and $\alpha$ be a big class. Then for any reflexive sheaf $\mathcal{E}$ on $X$, there exists a reflexive subsheaf $\mathcal{F}_{\alpha}$ of $\mathcal{E}$ such that $0\neq \mathcal{F}_{\alpha}\subsetneq \mathcal{E}$ and
    $$
    \mu_{\alpha}(\mathcal{F}_{\alpha})
    =\sup\{\mu_{\alpha}(\mathcal{F})\mid 0\neq \mathcal{F}\subsetneq \mathcal{E}:\hbox{a reflexive subsheaf}\}.
    $$
\end{lemm}

For the proof of Proposition \ref{max-slope}, the following lemma is essential. We recall that, for $\alpha, \beta \in H^{k,k}(X, \mathbb{R})$, the inequality $\alpha \geq \beta$ means that $\alpha - \beta$ is represented by a positive $(k,k)$-current.
\begin{lemm}[see also \cite{Cao}]\label{good basis}
Let $X$ be a compact complex manifold with a K\"{a}hler class $\omega$ and $\alpha \in H^{1,1}(X, \mathbb{R})$ be a big class on $X$. Let $\mathcal{E}$ be a reflexive sheaf on $X$. Then, there is a basis $(w_{i})_{i}$ of $H^{2n-2}(X, \mathbb{Q})$ such that
\begin{enumerate}
\item $\langle \alpha^{n-1} \rangle = \sum_{i}\lambda_{i}w_{i}$ for some $\lambda_{i} > 0$, and
\item each $w_{i}$ is represented by a strictly positive $(2n-2)$-current.
\end{enumerate}
\end{lemm}
\begin{proof}
We set a closed cone $P$ of $H^{2n-2}(X, \mathbb{R})$ as follows:
\begin{equation}
P := \{\varphi \in H^{2n-2}(X, \mathbb{R}) \mid \text{$\varphi$ is represented by a closed positive $2n-2$-current}\}.
\end{equation}
Since the internal of $P$, denoted by ${\rm{Int}}(P)$, is nonempty and open in $H^{2n-2}(X,\mathbb{R})$, we an choose a basis $(w_1,\ldots,w_s)$ of $H^{2n-2}(X,\mathbb{R})$ so that each $w_i$ lies in ${\rm{Int}}(P)\cap H^{2n-2}(X,\mathbb{Q})$. Since $\alpha$ is big, $\langle\alpha^{n-1}\rangle$ lies in ${\rm{Int}}(P)$. In fact, let $T$ be a K\"{a}hler current in $\alpha$ with $T\ge \varepsilon\omega$ where $\omega$ is a K\"{a}hler metric and $\varepsilon>0$. Then, by \cite{BEGZ}, we have
$$
\langle\alpha^{n-1}\rangle\ge \{\langle T^{n-1}\rangle\}\ge \omega^{n-1}.
$$
Hence $\eta:=\langle\alpha^{n-1}\rangle-\omega^{n-1}$ is represented by a positive current. Thus $\langle\alpha^{n-1}\rangle=\omega^{n-1}+\eta$ is represented by a strictly positive current. 
\end{proof}

\begin{proof}[Proof of Lemma \ref{max-slope}]
We show that the function $\mathcal{G}\mapsto \deg_{\alpha}(\mathcal{G})$ does not have accumulation points, where $\mathcal{G}$ is a reflexive subsheaf of $\mathcal{E}$ with $0\ne \mathcal{G}\subsetneq \mathcal{E}$.

Let $\omega$ be a K\"{a}hler class on $X$. Let $\mathcal{G} \subset \mathcal{E}$ be any nontrivial reflexive subsheaf. We choose a basis $(w_{i})$ of $H^{2(n-1)}(X, \mathbb{Q})$ as in Lemma \ref{good basis}. Then we have 
\begin{align}\label{eq6.15}
deg_{\alpha}(\mathcal{G}) 
&= \int_Xc_{1}(\mathcal{G})\wedge\langle\alpha^{n-1}\rangle \notag \\
&= \lambda_{1}\int_Xc_{1}(\mathcal{G})\wedge w_{1} + \cdots + \lambda_{s}\int_Xc_{1}(\mathcal{G})\wedge w_{s}. 
\end{align}
Here $\lambda_{i}>0$ are the coefficients of $\langle \alpha^{n-1} \rangle$ as in Lemma \ref{good basis}. By Lemma \ref{deg bound}, there is a constant $C > 0$ such that 
\begin{equation}
deg_{\alpha}(\mathcal{G}) \leq C
\end{equation}
for any $\mathcal{G}$.
It suffices to consider reflexive subsheaves $0\neq \mathcal{G} \subsetneq \mathcal{E}$ satisfying
\begin{equation}\label{eq6.17}
-C \leq deg_{\alpha}(\mathcal{G}) 
\end{equation}
as we are taking the supremum of the slope.
Again by Lemma \ref{deg bound}, we have
\begin{equation}\label{6.18}
\int_Xc_{1}(\mathcal{G})\wedge w_{i} \leq C' 
\end{equation}
for any $i$ and $\mathcal{G}$
since each $w_{i}$ is represented by a closed positive current and thus by the proof of Lemma \ref{deg bound}. Then, from $\lambda_{i}\ge0$, $(\ref{eq6.15})$ and $(\ref{eq6.17})$, we obtain
$$
-C\le deg_{\alpha}(\mathcal{G})\le C_1+\cdots +\lambda_s\int_Xc_1(\mathcal{G})\wedge w_s+\cdots+C_s
$$
and thus
\begin{equation}\label{eq6.19}
-C'' \leq \int_Xc_{1}(\mathcal{G})\wedge w_{i}
\end{equation}
 for any $i$ and $\mathcal{G}$.
We recall that $w_{i} \in H^{2n-2}(X, \mathbb{Q})$ and $c_{1}(\mathcal{G}) \in H^{2}(X, \mathbb{Z})$. Thus, by (\ref{6.18}) and (\ref{eq6.19}), we know the set 
$$
A :=\{\int_Xc_{1}(\mathcal{G})\wedge w_{i} \in \mathbb{R} \mid 1 \leq i \leq s, \text{ reflexive subsheaf $0\neq \mathcal{G} \subsetneq \mathcal{E}$} \hbox{ with }(\ref{eq6.17})\}
$$ 
is a finite set. Hence we can see the $\langle\alpha^{n-1}\rangle$-degree function $\deg_{\alpha}$ as a function on a finite set $A$, and thus there is a nontrivial reflexive subsheaf $\mathcal{F}_{\alpha}\subsetneq\mathcal{E}$ which attains the maximum of $\langle\alpha^{n-1}\rangle$-slope.
\end{proof}

Let $X$ be a compact K\"{a}hler manifold and $\alpha$ be a big class on $X$. For $k=1,\ldots,n$, we define 
$$
\mathcal{P}_k:=\{\beta\in H^{1,1}(X,\mathbb{R}) \hbox{ $\mid$ $\beta$ is big and $\langle\beta^k\rangle-\langle\alpha^k\rangle$ is represented by a positive $(k,k)$-current}\}.
$$
\begin{prop}\label{stability app}
    Let $X$ be a compact K\"{a}hler manifold and $\alpha$ be a big class on $X$. 
    If a reflexive sheaf $\mathcal{E}$ on $X$ is $\langle\alpha^{n-1}\rangle$-stable, then there exists $U_{\alpha}\subset \mathcal{P}_{n-1}$ a neighborhood of $\alpha$ in $\mathcal{P}_{n-1}$ such that 
    $\mathcal{E}$ is $\langle\beta^{n-1}\rangle$-stable for any $\beta\in U_{\alpha}$.
\end{prop}
\begin{proof}
    Let $\beta\in\mathcal{P}_{n-1}$ and set $\gamma:=\langle\beta^{n-1}\rangle-\langle\alpha^{n-1}\rangle$ which is represented by a positive $(n-1,n-1)$-current. Let $\mathcal{F}\subsetneq\mathcal{E}$ be any nontrivial reflexive subsheaf. Since $1\le \rk(\mathcal{F})\le \rk\mathcal{E}-1$, there is a constant $C>0$ independent of $\mathcal{F}$ and $\beta$ such that 
    \begin{equation}
        \frac{\deg(\mathcal{F}, \gamma)}{\rk\mathcal{F}}\le C\int_{X}\omega\wedge\gamma.
    \end{equation}
    by (\ref{deg bound eq}). Then we have
    \begin{align}
        \mu_{\beta}(\mathcal{E})-\mu_{\beta}(\mathcal{F})
        &= \mu_{\alpha}(\mathcal{E})-\mu_{\alpha}(\mathcal{F})+\left(\frac{\deg(\mathcal{E},\gamma)}{\rk\mathcal{E}}-\frac{\deg(\mathcal{F},\gamma)}{\rk\mathcal{F}} \right)\notag\\
        &\ge \mu_{\alpha}(\mathcal{E})-\mu_{\alpha}(\mathcal{F}_{\alpha})+\frac{1}{\rk\mathcal{E}}\left(\int_{X}c_1(\mathcal{E})\wedge\gamma-C\int_{X}\omega\wedge\gamma \right).
    \end{align}
    Here $\mathcal{F}_{\alpha}$ be a nontrivial reflexive subsheaf of $\mathcal{E}$ with maximal $\langle\alpha^{n-1}\rangle$-slope (see Lemma \ref{max-slope}).
    Since $\mathcal{E}$ is $\langle\alpha^{n-1}\rangle$-slope stable, we have $\mu_{\alpha}(\mathcal{E})-\mu_{\alpha}(\mathcal{F}_{\alpha})>0$.
    We set 
    \begin{equation}\label{stability app eq}
    U_{\alpha}:=\{\beta\in \mathcal{P}_{n-1}\mid 
    \mu_{\alpha}(\mathcal{E})-\mu_{\alpha}(\mathcal{F}_{\alpha})>\frac{1}{\rk\mathcal{E}}\int_X(C\omega-c_1(\mathcal{E}))\wedge(\langle\beta^{n-1}\rangle-\langle\alpha^{n-1}\rangle)\},
    \end{equation}
    Then we have $\mu_{\beta}(\mathcal{E})-\mu_{\beta}(\mathcal{F})>0$ for any $\beta\in U_{\alpha}$ and thus $\mathcal{E}$ is $\langle\beta^{n-1}\rangle$-stable for any $\beta\in U_{\alpha}$.
\end{proof}
\begin{corr}\label{stability app cor}
    Let $X$ be a compact K\"{a}hler manifold, $\omega$ be a K\"{a}hler class and $\alpha$ be a big class on $X$. 
    If a reflexive sheaf $\mathcal{E}$ on $X$ is $\langle\alpha^{n-1}\rangle$-slope stable, then $\mathcal{E}$ is also $\langle(\alpha+\varepsilon\omega)^{n-1}\rangle$-slope stable for sufficiently small $\varepsilon>0$.
\end{corr}

\subsection{Bogomolov-Gieseker inequality for big and nef class}
We recall that if $\alpha$ is nef and big, then $\langle\alpha^p\rangle=\alpha^p$ for any $p=1,\ldots,n$ (Proposition \ref{nef positive product}).
The Bogomolov-Gieseker inequality is a direct consequence of Proposition \ref{stability app cor}.
\begin{prop}\label{BG nef big}
Let $X$ be a compact normal space with a nef and big class $\alpha \in H^{1,1}_{BC}(X, \mathbb{R})$. Let $\mathcal{E}$ be a reflexive sheaf with $\rk\mathcal{E}=r$ on $X$ and $\pi : \widehat{X} \rightarrow X$ be a resolution of singularities of $X$. Suppose $\mathcal{E}$ is $\alpha^{n-1}$-slope stable. Then, the following Bogomolov-Gieseker inequality holds:
\begin{equation}\label{6.8}
(2rc_{2}(\pi^{[*]}{\mathcal{E}}) - (r-1)c_{1}(\pi^{[*]}{\mathcal{E}})^{2})\cdot(\pi^{*}\alpha)^{n-2}\geq 0.
\end{equation}
\end{prop}
\begin{proof}
Let $\eta$ be a K\"{a}hler class on $\widehat{X}$. By Corollary \ref{stability app cor}, the reflexive sheaf $\pi^{[*]}{\mathcal{E}}$ is $\alpha_{\varepsilon}:=(\pi^*\alpha+\varepsilon\omega)$-stable for any $\varepsilon>0$. Hence the Bogomolov-Gieseker inequality of $\pi^{[*]}{\mathcal{E}}$ holds with respect to $\alpha_{\varepsilon}$. Then the result (\ref{6.8}) follows by taking a limit $\varepsilon \rightarrow 0$. 
\end{proof}

\begin{lemm}\label{BG lemma}
Let $X$ be a compact normal space, $\alpha\in H^{1,1}_{BC}(X,\mathbb{R})$ be a nef and big class and $\mathcal{E}$ be a reflexive sheaf on $X$. If $X$ is smooth in codimension 2, then 
$$
\Delta(\mathcal{E})\alpha^{n-2}:=(2rc_2(\mu^{[*]}\mathcal{E})-(r-1)c_1(\mu^{[*]}\mathcal{E}))^2\cdot(\mu^*\alpha)^{n-2}
$$
is independent of the choices of resolutions $\mu:\widehat{X}\to X$. 
\end{lemm}
\begin{proof}
We remark that $c_2(\mu^{[*]}\mathcal{E})-c_2(\mu^*\mathcal{E}) \in H^{2,2}(\widehat{X},\mathbb{R})$ is supported in the $\mu$-exceptional divisor $D$. We also remark that $\dim_X(\mu(D))\le n-3$ since $\mu(D)=X_{\sing}$ and $\codim_XX_{\sing}\ge 3$. Therefore, by Proposition \ref{nef big singular exceptional}, we obtain
$$
(c_2(\mu^{[*]}\mathcal{E})-c_2(\mu^*\mathcal{E}))\cdot(\mu^*\alpha)^{n-2}=0.
$$
If we choose a further modification $\nu:Y\to \widehat{X}$ with $Y$ smooth, then 
$$
\nu^*\left(c_2(\mu^*\mathcal{E})\cdot(\mu^*\alpha)^{n-2}\right) 
=c_2(\nu^*\mu^*\mathcal{E})\cdot(\nu^*\mu^*\alpha)^{n-2}
$$
The same calculation works for the first chern class.
Then we can easily see that the RHS of $\Delta(\mathcal{E})\alpha^{n-2}$ is independent of the choices of resolutions $\mu:\widehat{X}\to X$.
\end{proof}

By Proposition \ref{BG nef big} and Lemma \ref{BG lemma}, we obtain the following.
\begin{corr}\label{BG nef big 2}
Let $X$ be a compact normal space, $\alpha\in H^{1,1}_{BC}(X,\mathbb{R})$ be a nef and big class and $\mathcal{E}$ be a reflexive sheaf on $X$. Suppose $X$ is smooth in codimension 2 and $\mathcal{E}$ is $\alpha^{n-1}$-stable. Then the Bogomolov-Gieseker inequality holds:
$$
\Delta(\mathcal{E})\alpha^{n-2}\ge 0.
$$
\end{corr}

We obtain the characterization of the equality on minimal projective varieties of general type. This is essentially due to \cite{Chen}. See also \cite{GKPT}. See Definition \ref{nonKahler locus singular} for the definition of the ample locus on singular spaces. The reader can consult \cite{Chen} about the Bogomolov-Gieseker inequality on compact normal K\"{a}hler spaces.
\begin{theo}\label{BG equality}
Let $X$ be a normal projective variety with log canonical singularities where $K_X$ is nef and big. Let $\mathcal{E}$ be a reflexive sheaf on $X$. Suppose $\mathcal{E}$ is $c_1(K_X)^{n-1}$-stable. If there exists a resolution $\pi:Y\to X$ such that $\pi^{[*]}\mathcal{E}$ satisfies the Bogomolov-Gieseker equality: $\Delta(\pi^{[*]}\mathcal{E})c_1(\pi^*K_X)^{n-2}=0$, then $\mathcal{E}$ is projectively flat on $\Amp(K_X)$. 
\end{theo}
\begin{proof}
By the base point free theorem in \cite{Fujino16}, we know $K_X$ is semiample. Therefore there is a birational morphism $\mu:X\to Z$ to a normal projective variety $Z$ with $K_Z$ ample and $K_X=\pi^*K_Z$.
By Theorem \ref{stability bimero}, the reflexive sheaf $\mu_{[*]}\mathcal{E}$ is $c_1(K_X)^{n-1}$-stable.
Let $\pi:Y\to X$ be a resolution with $Y$ smooth as in the statement,  we obtain a birational morphism $\mu\circ\pi:Y\to Z$ and $\pi^*K_X=(\mu\circ\pi)^*K_Z$.
\begin{center}
\begin{tikzpicture}[auto]
\node (y) at (0,1.5) {$Y$}; \node (x) at (0,0) {$X$}; \node (z) at (1.5,0) {$Z$};
\draw[->] (y) to node[swap] {$\pi$} (x); \draw[->] (y) to node {$\pi\circ\mu$} (z);
\draw[->] (x) to node {$\mu$} (z);
\end{tikzpicture}
\end{center} 
By \cite{Chen}, we obtain 
$$
\Delta(\pi^{[*]}\mathcal{E})c_1(\pi^*K_X)^{n-2}\ge\Delta(\mu_{[*]}\mathcal{E},h)\omega_Z^{n-2}\ge 0,
$$
where $\omega_Z$ is a K\"{a}hler metric in $c_1(K_Z)$ and $h$ is the admissible $\omega_Z$-HE metric in $\mu_{[*]}\mathcal{E}$. Since $\Delta(\pi^{[*]}\mathcal{E})c_1(\pi^*K_X)^{n-2}=0$ by assumption of this theorem, we have $\Delta(\mu_{[*]}\mathcal{E},h)\omega_Z^{n-2}= 0$. Hence $\mu_{[*]}\mathcal{E}$ is projectively flat on $Z_{\reg}$ by \cite{Chen}. Therefore, together with Proposition \ref{semiample-bimero}, we obtain that $\pi^{[*]}\mathcal{E}$ is projectively flat away from $\Exc(\pi\circ\mu)=E_{nK}(\pi^*K_X)$. Therefore we obtain that $\mathcal{E}$ is projectively flat on $\Amp(K_X)$.
\end{proof}

Fillip-Tosatti \cite{FT18} showed that any nef and big class on K3 K\"{a}hler surface is semiample (see Definition \ref{positive class}, Proposition \ref{semiample-bimero}). Thus we obtain the following complete result:
\begin{corr}
    Let $X$ be a K3 K\"{a}hler surface and $\alpha$ be a nef and big class on $X$. Suppose an $\alpha$-slope stable vector bundle ${E}$ on $X$ satisfies the Bogomolov-Gieseker equality: 
    $$
    \Delta({E})=2rc_2({E})-(r-1)c_1({E})^2=0.
    $$
    Then ${E}$ is projectively flat on $\Amp(\alpha)$.
\end{corr}


\begin{thebibliography}{99}
\bibitem[AKMW02]{AKMW}
D. Abramovich, K. Karu, K. Matsuki, J. Wlodarczyk, Torification and factorization of birational maps, J. Amer. Math. Soc. {\bf 15} (2002), no.~3, 531--572; MR1896232.
\bibitem[BS94]{BS}
S. Band and Y.~T. Siu, Stable sheaves and Einstein-Hermitian metrics, in {\it Geometry and analysis on complex manifolds}, 39--50, World Sci. Publ., River Edge, NJ, ; MR1463962.
\bibitem[BT82]{BT}
E. Bedford and B.~A. Taylor, A new capacity for plurisubharmonic functions, Acta Math. {\bf 149} (1982), no.~1-2, 1--40; MR0674165.
\bibitem[BCHM10]{BCHM}
C. Birkar, P. Cascini, D. Hacon and J. Mckernan, Existence of minimal models for varieties of log general type, J. Amer. Math. Soc. {\bf 23} (2010), no.~2, 405--468; MR2601039.
\bibitem[BH14]{BH14}
C. Birkar and Z. Hu, Polarized pairs, log minimal models, and Zariski decompositions, Nagoya Math. J. {\bf 215} (2014), 203--224; MR3263528.
\bibitem[Bou02]{Bou}
S. Boucksom, On the volume of a line bundle, Internat. J. Math. {\bf 13} (2002), no.~10, 1043--1063; MR1945706.
\bibitem[Bou04]{Bou2}
S. Boucksom, Divisorial Zariski decompositions on compact complex manifolds, Ann. Sci. \'Ecole Norm. Sup. (4) {\bf 37} (2004), no.~1, 45--76; MR2050205.
\bibitem[BDPP13]{BDPP}
S. Boucksom, J.-P. Demailly, M. Paun and T. Peternell, The pseudo-effective cone of a compact K\"ahler manifold and varieties of negative Kodaira dimension, J. Algebraic Geom. {\bf 22} (2013), no.~2, 201--248; MR3019449.
\bibitem[BEGZ10]{BEGZ}
S. Boucksom, P. Eyssidieux, V. Guedji, and A. Zeriahi, Monge-Amp\`ere equations in big cohomology classes, Acta Math. {\bf 205} (2010), no.~2, 199--262; MR2746347.
\bibitem[BFJ09]{BFJ}
S. Boucksom, C. Favre and M. Jonsson, Differentiability of volumes of divisors and a problem of Teissier, J. Algebraic Geom. {\bf 18} (2009), no.~2, 279--308; MR2475816.
\bibitem[Cao13]{Cao}
J. Cao, A remark on compact K\" ahler manifolds with nef anticanonical bundles and its applications. arXiv preprint arXiv:1305.4397 (2013).
\bibitem[Chen25]{Chen}
X. Chen, Admissible Hermitian–Yang–Mills connections over normal varieties. Mathematische Annalen (2025): 1-37.
\bibitem[CT15]{CT15}
T.~C. Collins and V. Tosatti, K\"ahler currents and null loci, Invent. Math. {\bf 202} (2015), no.~3, 1167--1198; MR3425388.
\bibitem[CT22]{CT22}
T.~C. Collins and V. Tosatti, Restricted volumes on K\"ahler manifolds, Ann. Fac. Sci. Toulouse Math. (6) {\bf 31} (2022), no.~3, 907--947; MR4452256.
\bibitem[DHP24]{DHP24}
O. Das, C.~D. Hacon and M. P\u aun, On the 4-dimensional minimal model program for K\"ahler varieties, Adv. Math. {\bf 443} (2024), Paper No. 109615, 68 pp.; MR4719824.
\bibitem[DHY23]{DHY23}
O. Das, C. Hacon, and J. I. Yáñez, MMP for Generalized Pairs on K\" ahler 3-folds. arXiv preprint arXiv:2305.00524 (2023).
\bibitem[Dem92]{Dem3}
J.-P. Demailly, Regularization of closed positive currents and intersection theory, J. Algebraic Geom. {\bf 1} (1992), no.~3, 361--409; MR1158622.
\bibitem[Dem1]{Dem4}
J-.P. Demailly, Complex Analytic and Differential Geometry. Book available at
https://www-fourier.ujf-grenoble.fr/~demailly/documents.html
\bibitem[Dem2]{Dem2}
J.-P. Demailly, Analytic Methods in Algebraic Geometry, Higher Education Press, Surveys of Modern Mathematics, Vol. 1, 2010, 231 pages. Book available at 
https://www-fourier.ujf-grenoble.fr/~demailly/documents.html.
\bibitem[Don87]{Don3}
S.~K. Donaldson, Infinite determinants, stable bundles and curvature, Duke Math. J. {\bf 54} (1987), no.~1, 231--247; MR0885784.
\bibitem[FT18]{FT18}
S. Filip and V. Tosatti, Smooth and rough positive currents, Ann. Inst. Fourier (Grenoble) {\bf 68} (2018), no.~7, 2981--2999; MR3959103.
\bibitem[Fuj16]{Fujino16}
O. Fujino, Basepoint-free theorem of Reid-Fukuda type for quasi-log schemes, Publ. Res. Inst. Math. Sci. {\bf 52} (2016), no.~1, 63--81; MR3452046.
\bibitem[GKP16]{GKP}
D. Greb, S. Kebekus and T. Peternell, Movable curves and semistable sheaves, Int. Math. Res. Not. IMRN {\bf 2016}, no.~2, 536--570; MR3493425.
\bibitem[GKPT19]{GKPT}
D. Greb, S. Kebekus, T. Peternell, B. Taji, The Miyaoka-Yau inequality and uniformisation of canonical models, Ann. Sci. \'Ec. Norm. Sup\'er. (4) {\bf 52} (2019), no.~6, 1487--1535; MR4061021.
\bibitem[GT17]{GT17}
D. Greb and M. Toma, Compact moduli spaces for slope-semistable sheaves, Algebr. Geom. {\bf 4} (2017), no.~1, 40--78; MR3592465
\bibitem[GZ17]{GZ}
V. Guedj and A. Z\'eriahi, {\it Degenerate complex Monge-Amp\`ere equations}, EMS Tracts in Mathematics, 26, Eur. Math. Soc., Z\"urich, 2017; MR3617346.
\bibitem[Gue16]{Gue}
H. Guenancia, Semistability of the tangent sheaf of singular varieties, Algebr. Geom. {\bf 3} (2016), no.~5, 508--542; MR3568336.
\bibitem[HP24]{HP24}
C. Hacon, and M. Paun, On the Canonical Bundle Formula and Adjunction for Generalized Kaehler Pairs. (2024), arXiv preprint arXiv:2404.12007.
\bibitem[Hiro75]{Hiro}
H. Hironaka, Flattening theorem in complex-analytic geometry, Amer. J. Math. {\bf 97} (1975), 503--547; MR0393556.
\bibitem[HP16]{HP16}
A. Höring, and T. Peternell, (2016). Minimal models for Kähler threefolds.  (2016), Inventiones mathematicae, 203(1), 217-264.
\bibitem[Kaw08]{Kaw}
Y. Kawamata, Flops connect minimal models, Publ. Res. Inst. Math. Sci. {\bf 44} (2008), no.~2, 419--423; MR2426353.
\bibitem[Kob87]{Kob}
S. Kobayashi, {\it Differential geometry of complex vector bundles}, Publications of the Mathematical Society of Japan Kan\^o{} Memorial Lectures, 15 5, Princeton Univ. Press, Princeton, NJ, 1987 Princeton Univ. Press, Princeton, NJ, 1987; MR0909698.
\bibitem[Laz04-1]{Laz}
R.~K. Lazarsfeld, {\it Positivity in algebraic geometry. I}, Ergebnisse der Mathematik und ihrer Grenzgebiete. 3. Folge. A Series of Modern Surveys in Mathematics, 48, Springer, Berlin, 2004; MR2095471.
\bibitem[Laz04-2]{Laz2}
R.~K. Lazarsfeld, {\it Positivity in algebraic geometry. II}, Ergebnisse der Mathematik und ihrer Grenzgebiete. 3. Folge. A Series of Modern Surveys in Mathematics, 49, Springer, Berlin, 2004; MR2095472.
\bibitem[LX15]{LX15}
B. Lehmann, and J. Xiao. Zariski decomposition of curves on algebraic varieties. arXiv preprint arXiv:1507.04316 (2015).
\bibitem[MM07]{MM}
X. Ma and G. Marinescu, {\it Holomorphic Morse inequalities and Bergman kernels}, Progress in Mathematics, 254, Birkh\"auser, Basel, 2007; MR2339952.
\bibitem[Mario13]{Mario}
M. Principato, (2013). Mobile product and zariski decomposition. arXiv preprint arXiv:1301.1477.
\bibitem[UY86]{UY}
K.~K. Uhlenbeck and S.-T. Yau, On the existence of Hermitian-Yang-Mills connections in stable vector bundles, Comm. Pure Appl. Math. {\bf 39} (1986), no. S, {\rm S}257--{\rm S}293; MR0861491.
\bibitem[Vu23]{Vu23}
Vu, D. V. Derivative of volumes of big cohomology classes. (2023), arXiv preprint arXiv:2307.15909.
\bibitem[Nystr19]{Nystr19}
D. Witt~Nystr\"om, Duality between the pseudoeffective and the movable cone on a projective manifold, J. Amer. Math. Soc. {\bf 32} (2019), no.~3, 675--689; MR3981985.
\end{thebibliography}



\end{document}